\documentclass[12pt]{article}

\usepackage{amsfonts,amsmath,amsthm,latexsym,amssymb,fullpage,amscd}
\usepackage[mathscr]{eucal}

\newcommand{\R}{\mathbb{R}}

\newtheorem{theorem}{Theorem}[section]
\newtheorem{lemma}[theorem]{Lemma}

\title{Fourier transforms of indicator functions, lattice point discrepancy, and the stability of integrals}
\author{Michael Greenblatt}
\date{\today}
     
\newcommand\blfootnote[1]{%
  \begingroup
  \renewcommand\thefootnote{}\footnote{#1}%
  \addtocounter{footnote}{-1}%
  \endgroup
}
\begin{document}
\maketitle
\begin{abstract} 
We prove sharp estimates for Fourier transforms of indicator functions of bounded open sets in $\R^n$ with real analytic boundary,
 as well as nontrivial lattice point discrepancy results. Both are derived from estimates on Fourier transforms of hypersurface measures.  Relations 
with maximal averages are discussed, connecting two conjectures of Iosevich and Sawyer from [ISa1].  We also 
prove a theorem concerning the stability under function perturbations of the growth rate of a real analytic function near a zero. This result is sharp in an appropriate sense. It implies a corresponding stability result for the local integrablity of negative powers of a real analytic function near a zero.

\end{abstract}

\section{Background and Statements of Results}

\subsection{Preliminaries}

Let $D$ be a bounded open set in $\R^{n+1}$ for some $n \geq 1$ whose boundary $\partial D$ is a real analytic hypersurface. By this we mean that if $x \in \partial D$ then there is a neighborhood $N_x$ of $x$ such that $N_x \cap \partial D$ is the graph of a real analytic function of some $n$ of the coordinate variables, and such that $N_x \cap D$ is the portion of $N_x$ that is on one side of this graph.

Our theorems will be stated in terms of an index we define as follows. Given $x_0 \in \partial D$, we perform a translation and rotation
taking $x_0$ to the origin after which the $\partial D$ can be represented as a graph of some $f(x)$ with $f(0) = 0$ and $\nabla f(0) = (0,...,0)$.
Then by resolution of singularities (see [AGuV] chapter 7 for a discussion of such facts and [G1] for elementary proofs of closely related statements) there is an $r_0 > 0$ such that there is a 
$g_{x_0} > 0$ and an integer $0 \leq k_{x_0} \leq n-1$ such that  if $r < r_0$ one has an asymptotic expansion of the following form, where $m$ 
denotes Lebesgue measure and $c_{r,x_0} > 0$.
\[ m(\{x: |x| < r, |f(x)| < \epsilon\}) = c_{r,x_0} \epsilon^{g_{x_0}} |\ln \epsilon|^{k_{x_0}} + o (\epsilon^{g_{x_0}} |\ln \epsilon|^{k_{x_0}})
\tag{1.1}\]
Note that by the relation between distribution functions and integrability of functions, $(1.1)$ implies that $|f(x)|^{-\delta}$ is integrable on some neighborhood of the origin for $\delta < g_{x_0}$ and is not integrable on any neighborhood
of the origin when $\delta \geq g_{x_0}$.

To give a better idea what the index $g_{x_0}$ means, suppose $f(x)$ has a zero of order $l \geq 2$ at the origin. Then there is some direction $v$ such
that the directional derivative $\partial_v^l f(x)$ is nonzero on a neighborhood of the origin. Then by the  measure Van der Corput lemma (Lemma 2.3) in the $v$ direction, there is a constant $C$ such that on some neighborhood of the origin, on each interval in the $v$ direction the measure of the points
for which $|f(x)| > \epsilon$ is at most $C\epsilon^{1 \over l}$. Thus integrating over the remaining $n-1$ directions we have that the exponent
$g_{x_0}$ is at least ${1 \over l}$. In the other extreme, if $f(x)$ has nonzero Hessian determinant at the origin, en elementary argument gives 
$g_{x_0} = {n \over 2}$. Thus $g_{x_0}$ is a measure of how flat $\partial D$ is at $x_0$, with a smaller value of $g_{x_0}$ corresponding to a flatter surface.

\noindent We define the index $g$ by
\[g = \inf_{x \in \partial D} g_x \tag{1.2}\]
Our results will be expressed in terms of the index $g$.  

There is also a natural geometric characterization of the index $g$ when $g \leq 1$. Let $T_{x_0}(\partial D)$ denote the tangent plane to $\partial D$ at $x_0$ and let $d(x, T_{x_0}(\partial D))$ denote the Euclidean distance 
from $x$ to this tangent plane. Let $\mu$ denote the Euclidean surface measure for the surface $\partial D$. Then by $(1.1)$ the index $g_{x_0}$ is equal
to $\sup \{h: \int_{\partial D \cap N} d(x, T_{x_0}(\partial D))^{-h} d\mu < \infty\}$ for any sufficiently small neighborhood $N$ of $x_0$.
Note that if $P$ is any other hyperplane containing $x_0$, then $\sup \{h: \int_{\partial D \cap N} d(x,P)^{-h} d\mu < \infty \}= 1$ for sufficiently small $N$, and if $P$ is a hyperplane not containing  $x_0$, then this supremum is infinity for sufficiently small $N$.
The above considerations combined with a compactness argument imply that if $g \leq 1$, then we have the following alternate definition.
\[g = \sup \{h:\int_{\partial D} d(x, P)^{-h} d\mu < \infty{\rm\,\, for\,\,all\,\,hyperplanes\,\,} P \} \tag{1.2'}\]

\subsection{Lattice point discrepancy}

We consider the situation where $D$ can be written in polar form as $\{(r, \omega) \in [0,\infty) \times S^n : r < h(\omega)\}$ where $h$ is a positive 
real analytic function on $S^n$. The classes of $D$ of Examples 1 and 2 described in section 1.4 give some examples of domains of this form. For $s > 0$ let $sD$ denote the dilate of $D$ given by $\{sx: x \in D\}$. Let $N(s)$ denote
the number of lattice points contained inside the closure of $sD$. Then $N(s) \sim s^{n+1}m(D)$ for large $s$, and a simple geometric argument further
gives that for some constant $C$ one has the estimate  $|N(s) - s^{n+1}m(D)| < Cs^n$. If $D$ were say a polyhedral region with rational
vertices instead of a domain 
with real analytic boundary, it is not hard to show that the flatness of the sides of $D$ would cause the exponent $n$ in $Cs^n$ to be best possible.

It turns out that the curvature of $\delta D$ in the situations at hand leads to stronger estimates on this discrepancy
function $|N(s) - s^{n+1}m(D)|$. In fact, well-known methods (see p.383-384 of [ShS]) imply  the following. Suppose $S$ is a smooth compact surface in 
$\R^{n+1}$ bounding an open set $U$ such that there
are constants $c, \delta_0 > 0$ such that whenever $\delta < \delta_0$, if $x \in U$ and $|y| < \delta$, then $x + y \in (1 + c\delta) U$. Then 
if the Euclidean surface measure $\nu$ on $S$ satisfies $|\hat{\nu}(\xi)| = O(|\xi|^{-\alpha})$, the lattice point discrepancy satisfies 
$N(s) - s^{n+1}m(U) = O(s^{n - {\alpha \over n + 1 - \alpha}})$.

Since the domains $D$ being considered here satisfy the above condition, as an immediate consequence of Theorem 1.3 of this paper in conjunction with
 a partition of unity on $\partial D$, we have the following. \vskip 0.2 in

\begin{theorem}

Suppose $D = \{(r, \omega) \in [0,\infty) \times S^n : r < h(\omega)\}$ where $h$ is a positive 
real analytic function on $S^n$. Let $g$ be as in $(1.2)$, and let $k = \max \{k_x: x \in \partial D$ with $g_x = g\}$.

\begin{enumerate}

\item If $g < {1 \over n + 1}$ and $k = 0$, then we have the estimate $N(s) - s^{n+1}m(D)
= O( s^{n - {g  \over n +1 - g}})$. 

\item If $g = {1 \over n + 1}$ or if $g < {1 \over n + 1}$ and $k > 0$, then for each $\epsilon > 0$ we have 
$N(s) - s^{n+1}m(D) =  O(s^{n - {g  \over n +1 - g} + \epsilon})$. 

\item If $g > {1 \over n + 1}$ then we have the estimate
$N(s) - s^{n+1}m(D) = O(s^{n - {1  \over n^2 + 2n}})$. 

\end{enumerate}
\end{theorem}

Finding the optimal exponent for the lattice point discrepancy is a notoriously difficult problem, and even for the unit disk in two dimensions the optimal 
exponent is unknown. 
The famous (and unsolved) Gauss circle problem is to show that for the disk one has an estimate  $N(s)- \pi s^2 = O( s^{{1 \over 2} + \epsilon})$ for any 
positive $\epsilon$; it was shown by G.H. Hardy [Ha1] [Ha2] that one does not have an estimate of the form $N(s)- \pi s^2 = O(  s^{{1 \over 2}})$. As of this writing the best exponent for which an estimate $N(s)- \pi s^2 = O(  s^\eta)$  is known is  $\eta = {517 \over 824}$, due to Bourgain and Watt [BWa].

 Due to the extensive history of the subject we can only give some of the highlights of the known lattice point discrepancy results. Much of the following
 history is taken from [IvKrKuNo].
For the disk in two dimensions, the exponent $\eta$ provided by the above-mentioned surface measure Fourier transform argument is 
$\eta = {2 \over 3}$, a result 
discovered by Sierpi\'nski  [Si]. One can improve on this exponent using expressions involving Bessel functions for discrepancy function $N(s)- \pi s^2$. There
were a number of papers in this direction, culminating in the paper [Ko] of Kolesnik which gave the exponent $\eta = {278 \over 429} + \epsilon$ for any
$\epsilon > 0$. In the 1980's, papers of Bombieri and Iwaniec [BoIw] and Iwaniec and Mozzochi [IwMo] developed a new approach to such problems 
which spurred further development. For some time, the best result expanding on this approach was the paper of Huxley [Hu3], giving the exponent
$\eta = {131 \over 208} + \epsilon$, but the recent above-mentioned preprint of Bourgain and Watt [BWa] provides the current best exponent of 
${517 \over 824}$.

For spheres in three dimensions, an early result of Landau [L] showed that $N(s) - s^3m(D)  = O(s^{3 \over 2})$. It was shown in several papers of Vinogradov, culminating in [Vi], that one has  $N(s) - s^3m(D)  = 
O(s^{{4 \over 3} + \epsilon})$. This exponent has been improved since then, for example in the paper of Heath-Brown [He] that gives an exponent of
${21 \over 16} + \epsilon$. It is known by an old result of Szeg\"o [Sz] that one can never get an exponent of $1$, providing a limit
to how far these results may be improved.

For spheres in four or more dimensions, the problem becomes much less difficult. One can show that for
 $n+1 \geq 4$ one has $N(s) - s^{n+1}m(D) = O(s^{n-1})$, and that this exponent is best possible. 
 We refer to [Kr2] for more information about the higher-dimensional situation. 
 
Most past results for domains other than spheres have been for convex domains. In the case of a convex domain in two dimensions whose boundary is
$C^2$ with nonvanishing curvature, van der Corput [Va1] showed that $N(s) - s^2m(D) = O(s^{2 \over 3})$. He later showed in [Va2] that if the boundary is in fact $C^{\infty}$, then
one necessarily has an exponent strictly less than ${2 \over 3}$, improved by Trifonov [Tr] to ${27 \over 41} + \epsilon$. Later, in the above-mentioned
paper of Huxley [Hu3] the exponent ${131 \over 208} + \epsilon$ was achieved for such domains with $C^3$ boundary. One will never get an exponent 
better than ${1 \over 2}$ that works for all $C^2$ convex domains with nonvanishing curvature, as was first proved by Jarnik [J1]. There are individual 
such domains where
one cannot improve Van der Corput's ${2 \over 3}$ exponent, proved by Jarnik in [J2].

For a convex domain in higher dimensions whose boundary has nonzero Gaussian curvature and satisfies an appropriate smoothness condition,
 Hlawka [Hl1][Hl2]  showed that one has
$N(s) - s^{n+1}m(D) = O(s^{n - 1 + {2 \over n + 2}})$. This result was improved over the subsequent decades. A noteworthy example is the paper of
M\"uller
[Mu] where it is shown that for $n = 2$ one can improve the exponent to $1 + {20 \over 43} + \epsilon$, for $n = 3$ one can improve it to 
$2 + {6 \over 17} + \epsilon$, and for $n \geq 4$ one can improve it to $n - 1 + {n + 5 \over n^2 + 3n + 4} + \epsilon$. 

There are also a variety of papers where the Gaussian curvature condition and/or smoothness condition on the boundary of the convex domain 
are relaxed, such as the papers [ISaSe1] [ISaSe2] [ISaSe3] [R3] [R4], whose proofs use harmonic analysis techniques.

We refer to the references  [Hu1] [Hu2] [IvKrKuNo] [Kr1] [Kr2] for further results on lattice point discrepancy problems.

\subsection{Maximal averages}

\noindent Let $S$ be a smooth hypersurface in $\R^{n+1}$, let $\phi(x)$ be a smooth cutoff function on $\R^{n+1}$ and let $\mu$ denote the Euclidean surface measure on $S$. Consider the operator
\[Mf = \sup_{s > 0} \bigg|\int_{\R^{n+1}} f(x  - s t)\,\phi(t)\,d\mu(t)\bigg|\]
Such operators are often referred to as  maximal averaging operators over hypersurfaces and there has been a lot of work on
these operators. The initial work was by Stein 
 [S2], where maximal averages over $n$-dimensional spheres were analyzed for $n \geq 2$ and $M$ was shown to be bounded on $L^p$ exactly 
when $p > {n + 1 \over n}$. The tricky case when $n = 1$ was later dealt with in Bourgain's paper [B], where boundedness of $M$ was shown to indeed
hold if and only if
$p > 2$. These results can be generalized to situations where $S$ is a hypersurface for which the Hessian determinant has positive rank, as was shown
in results of Sogge [So] and Greenleaf [Gr]. 

As for more general hypersurfaces, there have been a number of results. In the paper [I] of Iosevich, it was shown that the methods of [B] 
can be extended to finite-type 
curves in $\R^2$. For two-dimensional hypersurfaces in $\R^3$, Ikromov, Kempe, and Muller [IkKeM] proved a sharp result for $p > 2$ whenever the hypersurfaces satisfy a certain
transversality condition. The author also has a recent paper [G7] with results on this subject. The methods of [IkKeM] were later extended to various
$p < 2$ situations in [BuDeIkM]. The methods of [IkKeM] and [G7] have some commonalities
with those of this paper; they use resolution of singularities (in two dimensions) to prove oscillatory integral estimates related to surface measure
Fourier transform decay estimates, and these estimates are key to the proofs of the maximal operator boundedness theorems.  However, the resolution
of singularities methods of both papers are specific to two dimensions, so different resolution of singularities methods are needed in this paper. 

There also have been quite a few results in higher dimensions. The paper [SoS] by Sogge and Stein showed that if the Gaussian curvature of $S$
does not vanish to infinite order at any point, there is some $p < \infty$ for which $M$ is bounded on $L^p$. Optimal values of $p$ for which
$M$ is bounded on $L^p$ have been proven by the author under a nondegeneracy condition
on the Newton polyhedron [G5]. For convex hypersurfaces of finite line type, there have been many results. For example, there are results of
Cowling and Mauceri  [CoMa1] [CoMa2], Nagel, Seeger, and Wainger [NaSeW], Iosevich and Sawyer [ISa1], and Iosevich, Sawyer and Seeger [ISaSe4]. 

In [ISa1], two conjectures were made for the hypersurface case. The first conjecture is as follows.
 
\noindent {\bf Conjecture 1. [ISa1]} If $S$ is a smooth hypersurface and $p > 2$, then a necessary and sufficient condition for $M$ to be bounded on 
$L^p$ for every $\phi(x)$ is that $dist(x,V)^{-{1 \over p}}$ is locally integrable with respect to the measure $\mu$ for any hyperplane $V$ not containing the origin.

In [ISa2] the authors showed the necessity of this condition in relatively short order (for all $p > 1$ in fact), so the main issue is showing sufficiency for 
$p > 2$. The second conjecture in [ISa1] is as follows, extending a conjecture of Stein who considered the $\gamma = {1 \over 2}$ case. 
 
\noindent {\bf Conjecture 2. [ISa1]} If $S$ is a smooth hypersurface such that $|\hat{\mu}(\xi)|
\leq C(1 + |\xi|)^{-\gamma}$ for some $0 < \gamma \leq  {1 \over 2}$ then $M$ is bounded on $L^p$ for $p > {1 \over \gamma}$  for any $\phi(x)$.

Suppose  $S$ is a compact real analytic hypersurface with  $g \leq {1 \over n + 1}$ none of whose tangent planes contain
the origin. Then Conjecture 1 becomes the statement that a necessary and sufficient condition for $M$ to be bounded on $L^p$ for every $\phi(x)$ is
that $p > {1 \over g}$. So if Conjecture 2 holds, given $p_0 > {1 \over g}$ one can use a partition of unity in conjunction with Theorem 1.3 below and 
take $\gamma = {1 \over p_0}$ in Conjecture 2, leading to the statement that $M$ is bounded on $L^p$ for any $p > p_0$. Letting $p_0$ approach ${1 \over g}$ we then have that $M$ is bounded on $L^p$ for all $p > {1 \over g}$. 

Thus  Theorem 1.3 combined with Conjecture 2 imply that the sufficiency part of Conjecture 1 holds. Since the necessity part was shown in [ISa2], we have that Conjecture 2 implies Conjecture 1 for the surfaces at hand.

\subsection{Fourier transforms of indicator functions}

We let $D$ be a bounded open set in $\R^{n+1}$ with real analytic boundary, like before. As usual let $\chi_D(x)$ denote
the indicator function of $D$. We are interested in estimates of the following form for $|\xi| > 2$ say, where $C$ is a constant that can depend on the domain $D$.
\[|\widehat {\chi_D}(\xi)| \leq C |\xi|^{-\delta} (\ln |\xi|)^l \tag{1.3}\]
Here we would like $\delta > 0$ be optimal and $l \geq 0$ to be optimal for a given $\delta$. Since $D$ is a bounded domain, $|\widehat {\chi_D}(\xi)|$ is
a bounded function, so we are interested in seeing when $(1.3)$ holds for $|\xi| > N$ for some large $N$.

Using a partition of unity and simple integrations by parts transverse to the surface one can show $\delta \geq 1$, and the fact that
$\partial D$ is a non-flat real analytic surface can be used to show $\delta > 1$. The general heuristic is that the more curved the surface the better the 
exponent $\delta$ can be made. The most extreme case is when $\partial D$ has nonzero Gaussian curvature, where 
one can take $(\delta, l) = ({n \over 2} + 1,0)$.

Note that by Plancherel's theorem, since $D$ is bounded one has that  $\widehat {\chi_D}(\xi)$ is in $L^2(\R^{n+1})$. As a result, $|\widehat {\chi_D}(\xi)|$
decays as $o(|\xi|^{-{n+1 \over 2}})$ in a
certain averaged sense. But when the surface is not very curved at at  least one point on the boundary, there will be directions with a much slower decay rate.

This problem has been considered by a number of authors. We mention the early work of Randol [R1] [R2] (for $D$ not necessarily with real analytic boundary), which also proved lattice point discrepancy results using the Fourier transform methods described in section 1.2. [BraHoI] [Sv]  give further results on Fourier transforms of indicator functions for convex sets. When the set is convex one does not get a logarithmic factor, so the issue is finding the optimal power $|\xi|^{-\delta}$ in the Fourier transform decay rate.

As will be described in $(1.12)-(1.14)$, a straightforward integration by parts argument
connects decay rate estimates on $|\widehat {\chi_D}(\xi)|$ to decay rate estimates on the Fourier transform of  smooth surface measures on $\partial D$.
 We refer to the end of section 1.5 for a description of some prior work in the latter subject.
 
\noindent Our theorem regarding Fourier transforms of characteristic functions is as follows, where $C$ denotes a constant depending on the domain $D$.
Here $g$ is as in sections 1.1 and 1.2.

\begin{theorem}
\
\begin{enumerate}

\item Suppose $g < {1 \over n + 1}$. Then $g = g_x$ for at least one $x \in \partial D$. Let $k = \max \{k_x: x \in \partial D$ with $g_x = g\}$.
 Then for $|\xi| > 2$ the estimate $|\widehat {\chi_D}(\xi)| \leq C |\xi|^{-1 - g} (\ln |\xi|)^k$ is satisfied.
\item If $g = {1 \over n + 1}$, then part 1 holds except with an additional logarithmic factor: $|\widehat {\chi_D}(\xi)| \leq C |\xi|^{-1 - g} (\ln |\xi|)^{k+1}$.
\item If $g > {1 \over n + 1}$, then the estimate $|\widehat {\chi_D}(\xi)| \leq C |\xi|^{-1 - {1 \over n+1}}$ holds.
\item If $g < 1$, then one does not have an estimate $|\widehat{\chi_D}(\xi)| \leq C|\xi|^{-1 - \delta}$ for any $\delta > g$. Thus the exponent $-1 - g$ of parts 1 and 2 of this theorem is sharp.
\end{enumerate}
\end{theorem}

Although we will not prove it here, by rigorizing the heuristics at the end of this section one can actually show that if $x \in \partial D$ with $g_{x} < 1$ then
there even exists a constant $C'$ (depending on $D$) such that for sufficiently large $R$ one has $\sup_{|\xi| = R}|\widehat {\chi_D}(\xi)| \geq C' R^{-1 - g_x} (\ln R)^{k_x}$. Hence the estimate of part 1 of Theorem 1.2 is optimal.

\

\noindent {\bf Example 1.}

For positive integers $a_1,...,a_{n+1}$ let $D = \{x \in \R^{n+1}: \sum_{i=1}^{n+1} x_i^{2a_i} \leq 1\}$. Since each function $x_i^{2a_i}$ is convex, for
any $x$ and $y$ in $D$ and any $0 \leq t \leq 1$ one has $\sum_{i=1}^{n+1} (tx_i + (1- t)y_i)^{2a_i}  \leq \sum_{i=1}^{n+1} (tx_i^{2a_i} + (1 - t)y_i^{2a_i}) = t\sum_{i=1}^{n+1} x_i^{2a_i}  + (1 - t)  \sum_{i=1}^{n+1} y_i^{2a_i} < 1$. Hence $D$ is convex. 

Moreover, whenever $y \in \partial D$ with
more than one component $y_i$ nonzero, the surface $\partial D$ has at least one nonvanishing curvature at $y$. To see why this is the case, suppose $y_i$ and $y_j$ are
both nonzero. Then the two-dimensional cross section of $D$ in the $x_i$ and $x_j$ variables containing $y$ is a curve of the form $\{(x_i,x_j): x_i^{2a_i} + 
x_j^{2a_j} = c\}$ with $(y_i,y_j)$  on the portion of the curve with $x_i,x_j,c \neq 0$. A direct computation reveals that the curve has nonvanishing curvature at
such points. Therefore the curvature tangent to this curve is nonvanishing at $y$ and thus $g_y \geq {1 \over 2}$.

Next suppose $y \in \partial D$ is such that all components except one are zero. We focus on the case where $y_{n+1} \neq 0$, so that
$y = (0,...,0,1)$. Then since $x_{n+1} = (1 - \sum_{i=1}^n x_i^{2a_i})^{1 \over 2a_{n+1}}$ on $\partial D$ near $y$, one can represent 
$\partial D$ near $y$ as the graph of $f(x_1,...,x_n) = 1 - {1 \over 2a_{n+1}}b(x_1,...,x_n) \sum_{i=1}^n x_i^{2a_i}$, where $b(x)$ is real analytic with 
$b(0,...,0) = 1$. Thus the pair $(g_y, k_y)$ is characterized by the following holding for small $r > 0$.
\[ m(\{x: |x| < r, \sum_{i=1}^n x_i^{2a_i}  < \epsilon\}) = c_{r,y} \epsilon^{g_y} |\ln \epsilon|^{k_y} + o (\epsilon^{g_y} |\ln \epsilon|^{k_y})
\tag{1.4}\]
One can compute the exponent in $(1.4)$ directly using elementary methods, or using Newton polyhedra as in [AGuV]. The result is that $k_y = 0$ and
$g_y = \sum_{i=1}^n {1 \over 2a_i}$. 

The analogous result holds for the other $y$ for which all but one component is zero. So if $y_j$ is nonzero we have $g_y = \sum_{i \neq j} {1 \over 2a_i}$.
Thus the minimum of $g_y$ over all $y$ for which all but one component is zero is given by $g_0 = \sum_{i = 1}^{n+1} {1 \over 2a_i} -
{1 \over 2 \min_i a_i}$. So the index $g$ is given by $g_0$ if $g_0 \leq {1 \over 2}$, and some number at least ${1 \over 2}$ if $g_0 > {1 \over 2}$. 
If $g < {1 \over n + 1}$
as in part 1 of Theorem 1.2, we must be in the former situation and as a result $g = g_0 = \sum_{i = 1}^{n+1} {1 \over 2a_i} -
{1 \over 2 \min_i a_i}$.

\noindent {\bf Example 2.} 

Let $p(x_1,...,x_n)$ be a polynomial with $p(0) = 0$ and let $a$ be a positive integer such that $2a$ is greater than the degree of $p$. For a
small $\eta > 0$, let $D = \{x \in \R^{n+1}: \sum_{i=1}^{n+1} x_i^{2a} -  \eta\, p(x_1,...,x_n) = 1\}$. If $\eta$ is sufficiently small, then $\partial D$ is
domain with real analytic boundary, as can be seen  by observing that the outward directional derivative will be nonzero on $\partial D$
if $\eta$ is sufficiently small.

Since $p(0) = 0$, the point $y = (0,...,0,1)$ is on $\partial D$. Near this point, $\partial D$ satisfies 
\[ x_{n+1} = (1 +  \eta\, p(x_1,...,x_n) - \sum_{i=1}^{n} x_i^{2a})^{1 \over 2a} \tag{1.5}\]
As a result, near $(0,...,0,1)$ the surface $\partial D$ is represented as the graph of $f(x)$, where $f(x)$ is of the form
\[ f(x_1,...,x_n) = 1 + b(x_1,...,x_n) ({\eta \over 2a}\, p(x_1,...,x_n) - {1 \over 2a} \sum_{i=1}^{n} x_i^{2a}) \tag{1.6}\]
Here $b(x)$ is real analytic with $b(0) = 1$. Stated another way, for any polynomial $p(x)$ with $p(0) = 0$ one can find an arbitrarily high order perturbation
 $q(x)$ of a small multiple of $p(x)$ for which there exists a domain $D$ with $y = (0,...,0,1) \in \partial D$ near which $\partial D$ is the graph of
  $1 + b(x)q(x)$ near $y$, where $b(0) = 1$. This illustrates
that there is a wide range of possible local behavior for the surfaces and indices being considered in this paper.

In the case at hand, the pair $(g_y, k_y)$ is characterized by the following being true for small $r > 0$.
\[ m(\{x: |x| < r,\,\big|p(x_1,...,x_n) - {1 \over  \eta} \sum_{i=1}^{n} x_i^{2a}\big| < \epsilon\}) = c_{r,y} \epsilon^{g_y} |\ln \epsilon|^{k_y} + o (\epsilon^{g_y} |\ln \epsilon|^{k_y})
\tag{1.7}\]

\noindent {\bf Some motivation behind Theorem 1.2.} 

To give an idea of why $(1 + g,k)$ might be the optimal exponent pair for which $(1.3)$ holds, let $v$ be any direction, and suppose we are looking 
at $\widehat{\chi_D}(\xi)$ for $\xi$ in the $v$ direction. Let $x_0$ be any point on $\partial D$ for which $v$ is the normal direction to $\partial D$ at $x_0$. 
After a translation and rotation, we can assume that $x_0$ is the origin and that $v$ is the $(0,...,0,1)$ direction. Let $f(x)$ be such that the surface 
$\partial D$ is parameterized by $x_{n+1} = f(x_1,...,x_n)$ with $f(0) = 0$ and $\nabla f(0) = (0,...,0)$. For some small $r > 0$, let $A_r$ denote
the disk $\{x \in \R^{n}: |(x_1,...,x_n)| < r\}$ and let $\beta(x)$ be a nonnegative cutoff function on $\R$ supported on $(-r,r)$ and nonzero on a 
neighborhood of
$x = 0$. We examine the Fourier transform of $\chi_D(x) \chi_{A_r}(x_1,...,x_n)\beta(x_{n+1})$ in the $(0,...,0,1)$ direction, given by $I(t)$, where
\[ I(t) = \int_{-\infty}^\infty\bigg(\int_{|(x_1,...,x_n)| < r}\chi_D(x)\,dx_1...\,dx_n\bigg)\beta(x_{n+1})e^{-i tx_{n+1}}\,dx_{n+1} \tag{1.8}\]
Write $f(x) = f^+(x) - f^-(x)$ as the difference of its positive and negative parts, and let $c = m(\{x: |x| < r, f(x) < 0\})$.

If $x_{n+1} > 0$, the inner integral of $(1.8)$ is given by 
$c + m(\{x: |x| < r,\,0 < f^+(x) < x_{n+1}\})$, while if $x_{n+1} < 0$, the inner integral of $(1.8)$ is given by 
$c  - m(\{x: |x| < r,\,0 < f^-(x) < -x_{n+1}\})$. Thus $I(t)$ can be rewritten as
\[ c \int_{-\infty}^\infty  \beta(x_{n+1})e^{-i tx_{n+1}}\,dx_{n+1} \]
\[+ \int_0^{\infty}m(\{x: |x| < r,\,0 < f^+(x) < x_{n+1}\})\beta(x_{n+1})e^{-i tx_{n+1}}\,dx_{n+1} \]
\[ -\int_0^{\infty}m(\{x: |x| < r,\,0 <  f^-(x) < x_{n+1}\})\beta(x_{n+1})e^{i tx_{n+1}}\,dx_{n+1} \tag{1.9}\]
Because $\beta$ is smooth, the first integral in $(1.9)$ decays faster than $C|t|^{-N}$ for any $N$. For the second and third integrals, assuming 
$f^+$ and $f^-$ respectively are not identically zero near the origin, if $r$ is 
sufficiently small, analogously to $(1.1)$ we may write
\[  m(\{x: |x| < r,\, 0 < f^+(x) < \epsilon\}) = c_r^+ \epsilon^{g^+} |\ln \epsilon|^{k^+} + o (\epsilon^{g^+} |\ln \epsilon|^{k^+})
\tag{1.10a}\]
\[  m(\{x: |x| < r,\, 0 < f^-(x) < \epsilon\}) = c_r^- \epsilon^{g^-} |\ln \epsilon|^{k^-} + o (\epsilon^{g^-} |\ln \epsilon|^{k^-})
\tag{1.10b}\]
Then the growth rate $\epsilon^{g_0} |\ln \epsilon|^{k_0}$ of $(1.1)$ (letting $x_0 = 0$) is the slower of the two corresponding growth rates of $(1.10a)-(1.10b)$.
By resolution of singularities, it turns out that derivatives in $\epsilon$ of the left-hand expressions of $(1.9)$ have asymptotics given by the derivatives in
$\epsilon$ of the corresponding right-hand sides, with the same for higher derivatives. Thus the second and third integrals of $(1.9)$ have the same asymptotics in $t$ as the integrals 
\[c_r^+\int_0^{\infty}   x_{n+1}^{g^+} |\ln x_{n+1}|^{k^+} \beta(x_{n+1})e^{-i tx_{n+1}}\,dx_{n+1}\]
\[ c_r^- \int_0^{\infty}  x_{n+1}^{g^-} |\ln x_{n+1}|^{k^-} \beta(x_{n+1})e^{i tx_{n+1}}\,dx_{n+1}\tag{1.11}\]
By stationary phase one gets that for large $|t|$ the first integral is of order $|t|^{-1 - g^+}|\ln t|^{k^+}$ and the second integral is
of order $|t|^{-1 - g^-}|\ln t|^{k^-}$. When $g_0 < 1$, it can be shown that the coefficients of these terms are such that even if $g^+ = g^-$
and $k^+ = k^-$, these two leading terms will not cancel. As a result, the overall expression $(1.9)$ will have a decay rate in $t$ given
by the slower of these two decay rates, which is exactly the rate $(g_0, k_0)$. 

To find the decay rate of $\widehat{\chi_D}(\xi)$ in some arbitrary $v$ direction, one does a partition of unity in the integral defining $\widehat{\chi_D}(\xi)$.
The largest terms will be those terms that are localized near a boundary point $x_0$ of $D$ for which the normal to $\partial D$ is in the $v$ direction.
Using a rigorization of the above argument, such a term has magnitude comparable to $C|\xi|^{-1 - g_{x_0}}(\ln |\xi|)^{-k_{x_0}}$ for large $|\xi|$.
 Adding over all such terms of the partition of unity 
gives a bound of the maximum of several such terms. Since the index $g$ is the infimum of all $g_{x_0}$, we have that in the $v$ direction
 $\widehat{\chi_D}(\xi)$  satisfies the bounds of Theorem 1.2 part 1. This bound will be sharp since different terms coming from the partition of unity 
will not cancel one another out. This can be seen as follows. Note that a translation by some $a \in \R^{n+1}$ results in a factor of $e^{-i \xi \cdot a}$ on
the Fourier 
transform. As a result, the asymptotics corresponding to a term of the partition of unity centered at $a$ will have a coefficient of $e^{-i \xi \cdot a}$ 
on the asymptotics and several such terms can never cancel, even if they have the same rate of Fourier decay. As a result, the overall decay rate
of the Fourier transform will be exactly the minimum of the decay rates corresponding to the various terms from the partition of unity.

Note that the above heuristics only use consequences of resolution of singularities and not the Van der Corput lemma or other tools from harmonic 
analysis often used to analyze such integrals. The reason is that the above heuristics only apply in a single direction, and finding bounds that are
uniform over all directions is substantially more difficult. Thus to prove Theorem 1.2 more is needed. For this we turn to the connection between 
the above Fourier transforms of indicator functions and Fourier transforms of surface measures.

\subsection{Fourier transforms of hypersurface measures}

Let $x_0 \in \partial D$ and like above, without loss of generality we may assume that $x_0 = 0$ and $\partial D$ is the graph of some $f(x_1,...,x_n)$ 
with $f(0) = 0$ and $\nabla f(0) = 0$. Let $\psi(x)$ be a nonnegative cutoff function supported near the origin with $\psi(x) = 1$ on some neighborhood of
the origin. We look at $\widehat{\psi(x)\chi_D(x)}(\xi)$, given by
\[\int_{\R^n}\bigg(\int_{\{x: x_{n+1} > f(x_1,...,x_n)\}} \psi(x_1,...,x_{n+1})e^{-i\xi_{n+1} x_{n+1}}\,dx_{n+1}\bigg) e^{-i\xi_1x_1 -... -i\xi_nx_n}\,dx_1...\,dx_n \tag{1.12}\]
We change variables in the inner integral, turning $x_{n+1}$ into $x_{n+1} + f(x_1,...,x_n)$. Then for a new cutoff function $\psi_0(x_1,...,x_{n+1})$, 
$(1.12)$ becomes
\[\int_{\R^n}\bigg(\int_0^{\infty} \psi_0(x_1,...,x_{n+1})e^{-i\xi_{n+1}x_{n+1}}\,dx_{n+1}\bigg) e^{-i\xi_{n+1}f(x_1,...,x_n) - i\xi_1x_1 -... -i\xi_nx_n}\,dx_1...\,dx_n \tag{1.13}\]
Suppose $|\xi_{n+1}| < \sum_{i=1}^n |\xi_i|$. Then since $\nabla f(0) = 0$, if $\psi_0$ is supported on an appropriately small neighborhood of the origin,
one can integrate by parts in $(1.13)$ repeatedly in the $x_i$ variable for which $|\xi_i|$ is greatest and obtain a bound of $C_N|\xi_i|^{-N} \leq
C_N'|\xi|^{-N}$ for any $N$. Thus our interest is in the case when $|\xi_{n+1}| \geq \sum_{i=1}^n |\xi_i|$, which we henceforth assume.

We integrate by parts twice in the $x_{n+1}$ variable in $(1.13)$, integrating the $e^{-i\xi_{n+1}x_{n+1}}$ factor and differentiating 
$\psi_0(x_1,...,x_{n+1})$, leading to
\[{1 \over -i\xi_{n+1}}\int_{\R^n}\psi_0(x_1,...,x_n,0)e^{-i\xi_{n+1}f(x_1,...,x_n) - i\xi_1x_1 -... -i\xi_nx_n}\,dx_1...\,dx_n \]
\[+ {1 \over (\xi_{n+1})^2} \int_{\R^n}{\partial \psi_0 \over \partial x_{n+1}}(x_1,...,x_n,0)e^{-i\xi_{n+1}f(x_1,...,x_n) - i\xi_1x_1 -... -i\xi_nx_n}\,dx_1...\,dx_n \]
\[- {1 \over (\xi_{n+1})^2}\bigg(\int_0^{\infty} {\partial^2 \psi_0 \over \partial x_{n+1}^2}(x_1,...,x_{n+1})e^{-i\xi_{n+1}x_{n+1}}\,dx_{n+1}\bigg) e^{-i\xi_{n+1}f(x_1,...,x_n) - i\xi_1x_1 -... -i\xi_nx_n}\,dx_1...\,dx_n \tag{1.14}\]

Note that the second and third terms in $(1.14)$ are bounded by $C{1 \over |\xi_{n+1}|^2} \leq C'{1 \over |\xi|^2}$, and that the first term is
${1 \over -i\xi_{n+1}}$ times the Fourier transform of a surface measure supported on $\partial D$.
So since $|\xi_{n+1}| > c|\xi|$ in the situation at hand, if we can show that such surface measure Fourier transforms satisfy the estimates 
of Theorem 1.2 parts 1-3 except with exponents decreased by 1, then by using a partition of unity on $\partial D$ parts 1-3 of Theorem 1.2 will follow. That these estimates are satisfied follows
from the following theorem.

\begin{theorem}

Let $\mu$ denote the standard Euclidean surface measure on $\partial D$. Suppose $x_0 \in \partial D$, and let $g_{x_0}$ and $k_{x_0}$ be as in $(1.1)$. Then there is a neighborhood $N$ of $x_0$
in $\R^{n+1}$ such that if $\phi(x)$ is a smooth function supported on $N$, then for some constant $C > 0$ depending 
on both $D$ and $\phi$, for all $|\xi| > 2$ we have

\begin{enumerate}

\item If $g_{x_0} < {1 \over n + 1}$, then $\widehat{\phi(x) \mu}(\xi)\leq C|\xi|^{-g_{x_0}} (\ln |\xi|)^{k_{x_0}}$.
\item If $g_{x_0} = {1 \over n + 1}$, then $\widehat{\phi(x) \mu}(\xi)\leq C|\xi|^{-g_{x_0}} (\ln |\xi|)^{k_{x_0}+1}$.
\item If $g_{x_0} > {1 \over n + 1}$, then $\widehat{\phi(x) \mu}(\xi) \leq C |\xi|^{-{1 \over n + 1}}$.
\item If $g_{x_0} < 1$ and  $\phi(x)$ is a nonnegative function that is positive on a neighborhood of $x_0$, then one does not
have an estimate $\widehat{\phi(x) \mu}(\xi)\leq C|\xi|^{-h}$ for any $h > g_{x_0}$. Thus if $g_{x_0} \leq {1 \over n+ 1}$ the power of $|\xi|$ in
parts 1 and 2 is best possible. 

\end{enumerate}
\end{theorem}

Although we will not show it here, if $g_{x_0} < 1$ and $\phi(x)$ is a nonnegative function that is positive on a neighborhood of $x_0$, then similarly to Theorem 1.2
there is even a constant $C'$ such that for sufficiently large $R$ one has $\sup_{|\xi| = R}|\widehat{\phi(x) \mu}(\xi)| \geq C' R^{- g_{x_0}} (\ln R)^{k_{x_0}}$.

Suppose we have translated and rotated coordinates so that $x = 0$, $f(x) = 0$, and $\nabla f(0) = 0$. We focus our attention on the following expression for $\widehat{\phi(x) \mu}(\xi)$.
\[\widehat{\phi(x) \mu}(\xi) = \int_{\R^n}\phi(x_1,...,x_n) e^{-i\xi_{n+1}f(x_1,...,x_n) - i\xi_1x_1 -... -i\xi_nx_n}\,dx_1...\,dx_n
\eqno (1.15) \]

Using a rigorization of the heuristics at the end of section 1.4 (see chapter 6 of [AGV] for more information), when $\xi_i = 0$ for all $1 \leq i \leq n$, the integral in $(1.15)$ can be shown
to be bounded by $C|\xi|^{-g_{0}} (\ln |\xi|)^{k_{0}}$ for $|\xi| > 2$. Thus proving Theorem 1.3 is equivalent to showing this bound is stable under 
linear perturbations of the phase. By shrinking the neighborhood of $x_0$ we are working in if necessary, we can always assume that given $\eta > 0$ we 
have $\sum_{i=1}^n|\xi_i| < \eta|\xi|$, for when $\sum_{i=1}^n|\xi_i| \geq \eta|\xi|$ there is a neighborhood of $0$ on which repeated integrations by
parts give a bound of $C_N|\xi|^{-N}$ for any $N$. 

Thus the goal is to understand when the bound when $\xi_i = 0$ for all  $1 \leq i \leq n$ is stable under 
small linear perturbations of the phase. In the $n=1$ case, simply using the Van der Corput lemma for $k$th derivatives, where $k$ is the order of the
zero of $f(x)$ at 
$x = 0$, will do the trick. For arbitrary $n$, in this paper we will use resolution of singularities to simultaneously ``untangle'' the zeroes of $n+1$ functions in the phase to reduce to 
a situation where the functions are effectively monomials and the Van der Corput lemma for some $k$th derivative, with $k \leq n + 1$, can be used in one direction. The author used a similar philosophy in [G6]. Because we may need as many as
$n+1$ derivatives in the Van der Corput lemma, we only get the best possible estimate when $g_{x_0} < {{1 \over n + 1}}$, with an additional logarithmic 
factor if $g_{x_0} = {{1 \over n + 1}}$. This also accounts for the exponent ${1 \over n + 1}$ in part 3 of Theorems 1.2 and 1.3. The author does not know
if this boundary value $g_{x_0} = {{1 \over n + 1}}$ can be improved with an additional argument, such as an examination of the resolution of singularities
process being used rather than simply using a resolution of singularities theorem as we are doing in this paper.

There is an extensive literature concerning the decay rate of Fourier transforms of smooth hypersurface measures, which in view of integrations by parts such
as in $(1.12)-(1.14)$ imply corresponding results on the decay rate of Fourier transforms of indicator functions. If the surface has nonvanishing 
Gaussian curvature of $x_0$, then as is well known one has an estimate $\widehat{\phi(x) \mu}(\xi) \leq C|\xi|^{-{n \over 2}}$.
We refer to Chapter 8 of [S] for details and further results.  When the surface has at least one vanishing curvature, one can use 
similar methods to show one has  $\widehat{\phi(x) \mu}(\xi) \leq C|\xi|^{-{1 \over 2}}$.
One popular class of surfaces to study has been the convex hypersurfaces considered in [BrNaW] [CoDiMaM] [CoMa2] [NaSeW] [Sc]. It should be pointed out
that the paper [BrNaW] characterizes the Fourier decay rate for smooth convex hypersurfaces of finite line type in terms of the height of 
certain balls, and this characterization is readily seen to be equivalent (for such surfaces) to the characterization in terms of $g$ in Theorem 1.3.

\subsection{Stability of integrals}

Suppose $q(x)$ is a real analytic function defined on a neighborhood of the origin in $\R^n$ with $q(0) = 0$. Then as in $(1.1)$, there is an $r_0 > 0$
and a pair $(h, l)$ with $h > 0$ and $0 \leq l \leq n-1$ an integer such that if $0 < r \leq r_0$ we have an asymptotic development
\[ m(\{x: |x| < r,\,|q(x)| < \epsilon\}) = c_r \epsilon^h |\ln \epsilon|^l + o (\epsilon^h |\ln \epsilon|^l)
\tag{1.16}\]
A natural question is ask is how perturbing the function $q(x)$ affects the growth rate of $q(x)$. For example, we may ask for which $s(x)$ 
do we have an estimate of the following form  for some $r_1 < r_0$.
\[ m(\{x: |x| < r_1,\,|q(x) + s(x)| < \epsilon\})  \leq C \epsilon^h |\ln \epsilon|^l \tag{1.17}\]
Ideally, we would like the estimate $(1.17)$ to hold over all $s(x)$ in a large class of real analytic  perturbation functions that are small in an appropriate
sense. 

Besides its inherent interest, this question arises in complex geometry when one considers the analogous question for perturbations of analytic  
functions of several complex variables.
For example, a strong perturbation result was proved in [DK] and then used to prove the existence of K\"ahler-Einstein metrics on a class of Fano orbifolds, generalizing and simplifying earlier work of Nadel [N]. Significant work in this area has also been done by Phong and Sturm, such as in [PSt] and 
subsequent papers. The papers [DK] and [PSt] show that for (complex analytic) perturbations $s(x)$, there are $r_2 > r_1 > 0$ such that equation $(1.17)$ holds if $\sup_{|x| \leq r_2} |s(x)|$ is sufficiently small.

One gets much different behavior in the
real case, as can be seen by the example $q(x_1,x_2,x_3) = x_1^2 + x_2^2 + x_3^2$. Then $(h,l) = (3/2,0)$, but any small linear perturbation of 
$q(x_1,x_2,x_3)$ will have its index pair $(h,l)$ given by $(1,0)$. This phenomenon is not restricted to linear perturbations. If one takes 
$q(x_1,x_2,x_3) = x_1^4+ x_2^2+ x_3^2$, then it can be shown that $(h,l) = (5/4,0)$, while for $x_1^4 -  tx_1^2 + x_2^2+ x_3^2$,
if $t > 0$ one has $(h,l) = (1,0)$. This example is worked out in [G2] and is derived from a related example of
Varchenko [V].

Examples such as the above rule out general results in the real case that are analogous to the ones above in the complex case.  But there are classes of real-analytic functions for which the growth rate can only improve under perturbations. The simplest example occurs when $n = 1$. Suppose
$q(x)$ is a real analytic function of one variable defined near the origin with $q(0) = 0$. Let $m$ denote the order of the zero of $q(x)$ at $x = 0$. Then
$h = {1 \over m}$ and $l = 0$ in $(1.16)$. Let $I$ be an interval centered at $x=0$ and a $\delta > 0$ such that $|{d^m q \over dx^m}| > \delta$ on $I$. Then if $q(x) + s(x)$ is a perturbation
of $q(x)$ such that $|{d^m (q + s) \over dx^m}| > {\delta \over 2} $ on $I$, one can use the measure Van der Corput lemma, Lemma 2.3, to conclude that
there is a constant $C$ independent of $r$ such that $m(\{x \in I: |q(x) + s(x)| < \epsilon\}) < C\epsilon^{1 \over m}$.  Thus the estimate $(1.17)$
holds over all such perturbations with a uniform constant.

In two dimensions it was proved by Karpushkin [Ka1][Ka2] using versal deformation theory that one obtains a strong stability result similar to those for 
complex analytic functions described above. This was reproven up to endpoints in [PSSt] and they also showed that in three dimensions such a result holds
if $h < {2 \over N}$, where $N$ is the order of the zero of $q(x)$ at the origin. But results in higher dimensions have been hard to come by. We will 
show that if one restricts to perturbations coming from a finite-dimensional space of functions, one can get results that hold in any dimension. The 
maximum dimension of such a space of perturbation functions will be sharp in an appropriate sense. Our theorem is as follows.

\begin{theorem}

\
\begin{enumerate}
\item Write $h = {1 \over m + \gamma}$, where $m$ is a nonnegative integer and $0 <  \gamma \leq   1$. Let $s_1(x),...,s_m(x)$ be linearly independent functions,
each real analytic on a neighborhood of the origin and satisfying $s_i(0) = 0$.

Then there are positive constants $C, \delta,$ and $r_1 < r_0$ depending on $q(x)$ and the $s_i(x)$ such that for each $s(x)$ of the form $\sum_{i=1}^m \eta_i s_i(x)$ 
with each $|\eta_i| < \delta$ one has
\[ m(\{x: |x| < r_1,\,|q(x) + s(x)| < \epsilon\})  \leq C \epsilon^h |\ln \epsilon|^l \tag{1.18}\]
Consequently, if $t > 0$ is such that $\int_{\{x: |x| < r_0\}} |q(x)|^{-t} < \infty$, then there is a constant $D_t$ such that for each such $\sum_{i=1}^m \eta_i s_i(x)$ we have  $\int_{\{x: |x| < r_1\}} |q(x) + s(x)|^{-t} < D_t$
\item The number $m$ of functions in part 1 is sharp in the sense that if $t$ is any positive rational number and one writes $t = {1 \over m + \gamma}$ with 
$0 \leq \gamma < 1$, $m$ a nonnegative integer, then there exists a $q(x)$ for which $h = t$, real analytic functions $s_1(x),...,s_{m+1}(x)$ with each $s_i(0) = 0$, 
and an $h' < h$ such that there are no $\delta, r_1$ and $C$ for which every $s(x)$ of the form $\sum_{i=1}^{m+1}\eta_i s_i(x)$ with each 
$|\eta_i| < \delta$ satisfies
\[ m(\{x: |x| < r_1,\,|q(x) + s(x)| < \epsilon\})  \leq C \epsilon^{h'} \tag{1.19}\]
Here $q(x)$ and the $s_i(x)$ are functions of a number of variables that depends on $t$.

\end{enumerate}
\end{theorem}

Note that when $h = {1 \over m}$ is the reciprocal of an integer, there is a difference of two in the number of perturbation functions in the two parts of Theorem 1.4 
instead of one. This is because when $\gamma = 0$, the arguments of the proof of part 1 lead to an additional factor of $|\ln \epsilon|$ in $(1.18)$.
Therefore the case $\gamma = 0$ must be excluded from the statement of part 1. Also, since in part 2) the $s_i(x)$ are functions of a number of variables that 
depends on $t$, it does not preclude analogues of the three-dimensional theorem of [PSSt] from holding in higher dimensions.

The methods we will use to prove part 1 of Theorem 1.4 will be quite similar to those used to prove Theorem 1.3 as the arguments in the proof of Theorem 1.3 apply
to nonlinear perturbation functions  equally well as linear perturbation functions. In place of the oscillatory integral Van der Corput lemmas, Lemma 2.1
 and Lemma 2.2, we will use the measure version, Lemma 2.3.

\section{The proof of parts 1-3 of  Theorems 1.2 and 1.3.}

As explained above Theorem 1.2, parts 1-3 of Theorem 1.2  follow from parts 1-3 of Theorem 1.3, so that is what we will prove in this section.

\subsection{Preliminary lemmas.}

We will make essential use of Van der Corput-type lemmas in our arguments. If $p > 1$, we use the traditional Van der Corput lemma 
(see p. 334 of [S1]):\\

\begin{lemma}
Suppose $h(x)$ is a real-valued $C^k$ function on the interval $[a,b]$ such that $|h^{(k)}(x)| > A$ on $[a,b]$ for
some $A > 0$. Let $\phi(x)$ be $C^1$ on $[a,b]$. 

\noindent If $k \geq 2$ there is a constant $c_k$ depending only on $k$ such that
\[\bigg|\int_a^b e^{ih(x)}\phi(x)\,dx\bigg| \leq c_kA^{-{1 \over k}}\bigg(|\phi(b)| + \int_a^b |\phi'(x)|\,dx\bigg) \tag{2.1}\]
If $k =1$, the same is true if we also assume that $h(x)$ is $C^2$ and $h'(x)$ is monotone on $[a,b]$. 
\end{lemma}

\noindent  When $p = 1$, we will use the following variant of Lemma 2.1 that holds when $k = 1$.\\

\begin{lemma} (Lemma 2.2 of [G3].) Suppose the hypotheses of Lemma 2.1 hold for  $k = 1$, except  instead of assuming that $h'(x)$ is monotone on $[a,b]$ we assume that $|h''(x)|
< {B  \over(b-a)}A$ for some constant $B > 0$. Then we have
\[\bigg|\int_a^b e^{ih(x)}\phi(x)\,dx\bigg| \leq  A^{-1}\bigg((B+2) \sup_{[a,b]}|\phi(x)| + \int_a^b |\phi'(x)|\,dx \bigg) \tag{2.2}\]
\end{lemma}

We will also make use of the following measure version of the Van der Corput lemma. We refer to [C] for a proof.

\begin{lemma} Suppose $k \geq 1$ and $f(t)$ is a $C^k$ function on an interval $I$ such that $\displaystyle \bigg|{\partial^k f \over \partial t^k}\bigg|> A\displaystyle$ on $I$, where 
$A > 0$. There is a constant $B_k$ depending on $k$ only such that for each $\epsilon > 0$ the following holds.
\[m(\{t \in I: |f(t)| < \epsilon\}) \leq B_k A^{-{1 \over k}}\epsilon^{1 \over k}\]
\end{lemma}

\
\noindent We next prove a compactness lemma needed for our arguments. In the following, for mutually orthogonal unit vectors
 $v_1,...,v_{n+1} \in \R^{n+1}$ and $\epsilon > 0$, we let
$ A_{v_1,...,v_{n+1},\epsilon} = \{c_1v_1 + ... + c_{n+1}v_{n+1}: |c_1| \leq 1, |c_{i+1}| \leq \epsilon|c_i|$ for all $i \geq 1\}$.

\vskip 0.2 in
\begin{lemma}
Suppose  for each set of mutually orthogonal unit vectors $v_1,...,v_{n+1}$ in $\R^{n+1}$ the set $B_{v_1,...,v_{n+1}}$ is of the form
$ A_{v_1,...,v_{n+1},\epsilon}$  for some $\epsilon > 0$. Then there is a finite collection $\{B_i\}_{i=1}^N$ where each $B_i$ is some 
$B_{v_1,...,v_{n+1}}$, such that $\{x: |x| \leq 1 \} \subset \cup_{i=1}^N B_i$.
\end{lemma}

\begin{proof} We induct on the dimension. For $n = 1$ the lemma follows immediately from the compactness of the unit circle, so we assume the result
is known for $n-1$ and seek to prove it for some $n \geq 2$. For a fixed $v_1$, for some $ \epsilon_{v_2,...,v_{n+1}} > 0$ a given $B_{v_1,...,v_{n+1}}$ contains $C_{v_2,...,v_{n+1}} 
 = \{v_1 + c_2v_2 + ... + c_{n+1}v_{n+1}: |c_2| \leq \epsilon_{v_2,...,v_{n+1}}, |c_{i+1}| \leq  \epsilon_{v_2,...,v_{n+1}}|c_i|$ for all $i \geq 2\}$. Let $D_{v_2,...,v_{n+1}}$ denote the dilate of $C_{v_2,...,v_{n+1}}$ by a factor of  $(\epsilon_{v_2,...,v_{n+1}})^{-1}$
 in the hyperplane $\{x: x \cdot v_1 = 1\}$, where the dilation is centered at $v_1$. Then the sets $D_{v_2,...,v_{n+1}}$ satisfy the 
hypotheses of this lemma in this hyperplane in one lower dimension. Thus we may apply the induction hypothesis to conclude that there is a finite collection 
of $D_{v_2,...,v_{n+1}}$ covering $\{x: x \cdot v_1 = 1, |x - v_1| \leq 1\}$.
 
Going back to the sets $C_{v_2,...,v_{n+1}}$, we see that there is therefore a finite collection of sets $C_{v_2,...,v_{n+1}}$ covering a set of the form
$\{x: x \cdot v_1 = 1, |x - v_1| \leq \delta \}$, where $\delta >0$. Since $C_{v_2,...,v_{n+1}} \subset B_{v_1,...,v_{n+1}}$, this means the corresponding finite
collection of $B_{v_1,...,v_{n+1}}$ also cover $\{x: x \cdot v_1 = 1, |x - v_1| \leq \delta \}$. But $B_{v_1,...,v_{n+1}}$ is defined so that
if $(x_1,...,x_{n+1}) \in  B_{v_1,...,v_{n+1}}$ then so is $(tx_1,...,t x_{n+1})$ for any $0 \leq  t \leq 1$. Hence this finite list of $B_{v_1,...,v_{n+1}}$ cover
a cone containing $v_1$. 

The above argument applies for any $v_1$ with $|v_1| = 1$, so by compactness of $S^n$, there is a finite list of $B_{v_1,...,v_{n+1}}$ covering all of
$\{x: |x| \leq 1 \}$. This concludes the proof of the lemma. \end{proof}

\subsection {\bf The proof of parts 1-3 of Theorem 1.3.}

As before, after rotating and translating if necessary we may assume that $x_0$ in Theorem 1.3 is the origin and that the vector $(0,...,0,1)$ is normal to 
$\partial D$ there. Writing $\partial D$ as the graph of some real analytic function $f(x_1,...,x_n)$  with $f(0) = 0$ and $\nabla(0) = 0$, one can 
write the Fourier transform of $\phi(x)\mu$ as follows:
\[\widehat{\phi(x)\mu}(\xi) = \int_{\R^n}\phi_0(x_1,...,x_n) e^{-i\xi_{n+1}f(x_1,...,x_n) - i\xi_1x_1 -... -i\xi_nx_n}\,dx_1...\,dx_n \tag{2.3}\]
Here $\phi_0(x_1,...,x_n) = \phi(x_1,...,x_n,f(x_1,...,x_n))$.
As before, we may assume that $|\xi_{n+1}| > \sum_{i=1}^n |\xi_n|$ since by repeated integrations by parts one has arbitrarily fast decay
 in the remaining directions.
 
Write $\xi = |\xi| \xi'$, where $|\xi'| = 1$. By Lemma 2.4, in order to prove Theorem 1.3 it suffices to show that for each $v_1,...,v_{n+1}$ there is
some $\epsilon > 0$ (depending on $v_1,...,v_{n+1}$) such that the estimates of Theorem 1.3 hold for $\xi' \in A_{v_1,...,v_{n+1},\epsilon} $. To this
end, for a given $v_1,...,v_{n+1}$ we write $\xi' = c_1v_1 + ... + c_{n+1}v_{n+1}$ where $|c_1| \leq 1, |c_{i+1}| \leq \epsilon|c_i|$ for all $i \geq 1$, 
where $\epsilon$ will be determined by our arguments. Since $|\xi'| = 1$ and $\epsilon$ will be small, we will always have ${1 \over 2} \leq |c_1| \leq 1$.

 Observe that the phase in $(2.3)$ is given by $\xi \cdot (x_1,...,x_n, f(x_1,...,x_n))$, which can be rewritten as
$|\xi| \sum_{i=1}^{n+1} c_i v_i \cdot (x_1,...,x_n, f(x_1,...,x_n))$. Thus it makes sense to define $g_i(x) = v_i \cdot (x_1,...,x_n, f(x_1,...,x_n))$, in which
case $(2.3)$ becomes
\[\widehat{\phi(x)\mu}(\xi) = \int_{\R^n}\phi_0(x) e^{-i|\xi|(c_1g_1(x) + ... c_{n+1}g_{n+1}(x))}\,dx_1...\,dx_n \tag{2.4}\]

We next apply resolution of singularities to the functions $g_1(x),...,g_{n+1}(x)$ simultaneously. For our purposes it will not matter exactly which resolution of singularities procedure we use, and for example Hironaka's famous work [H1] [H2] more than suffice. By such resolution of singularities, there is a neighborhood $N$ of 
the origin such that if $\alpha(x)$ is a smooth function supported in $N$ then $\alpha(x)$ can be written as $\sum_{j=1}^M \alpha_j(x)$ 
 such that for each $j$ there is a real analytic variable change $\Psi_j$ such that each $\alpha_j(\Psi_j(x))$ is smooth and on a neighborhood of the support of 
$\alpha_j(\Psi_j(x))$ each function
$g_i \circ \Psi_j(x)$ can be written in the form $a_{ij}(x)m_{ij}(x)$, where $a_{ij}(x)$ is real analytic and does not vanish on the support of $\alpha_j(\Psi_j(x))$
and $m_{ij}(x)$ is a nonconstant monomial $x_1^{\alpha_{ij1}}...x_{n}^{\alpha_{ijn}}$. By resolving the singularities of each $g_{i_1}(x) - 
g_{i_2}(x)$ for $i_1 \neq i_2$ at the same time, we can ensure that in the resolved coordinates, for a given $j$ the multiindices $\alpha_{ij} = (\alpha_{ij1},...,\alpha_{ijn})$ 
are lexicographically ordered (not necessarily strictly) in some ordering. 

We let $\alpha(x)$ be such a function that is equal to 1 on a neighborhood $N'$ of the origin, and without loss of generality we may assume that 
$\phi_0(x)$ is supported on $N'$. Writing
$\phi_j(x) = \phi_0(x)\alpha_j(x)$, we have that $\sum_{j=1}^M\phi_j(x) = \phi_0(x)$, and each $\phi_j(\Psi_j(x))$ is smooth.

We write $\widehat{\phi(x)\mu}(\xi) = \sum_{j=1}^N b_j(\xi)$ in accordance with the decomposition $\phi_0(x) =\sum_{j=1}^M \phi_j(x) $ as
inserted in $(2.4)$, 
and then perform the coordinate changes $\Psi_j(x)$  obtaining
\[ b_j(\xi) = \int_{\R^n}\phi_j(\Psi_j(x)) Jac_j(x)e^{-i|\xi|(c_1a_{1j}(x)m_{1j}(x) + ... c_{n+1}a_{n+1\,j}(x)m_{n+1\,j}(x))}\,dx_1...\,dx_n \tag{2.5}\]
Here $Jac_j(x)$ denotes the Jacobian determinant of the coordinate changes $\Psi_j(x)$.  Our resolution of singularities procedure can be done so
 that $Jac_j(x)$ is also comparable to a monomial, which will be useful in our analysis. We let $\bar{\phi}_j(x)$ denote $ \phi_j(\Psi_j(x))$, so
that  $\bar{\phi}(x)$ is a smooth function and $(2.5)$ becomes
\[ b_j(\xi) = \int_{\R^n}\bar{\phi}_j(x)Jac_j(x)e^{-i|\xi|(c_1a_{1j}(x)m_{1j}(x) + ... c_{n+1}a_{n+1\,j}(x)m_{n+1\,j}(x))}\,dx_1...\,dx_n \tag{2.5'}\]
Here the functions $a_{ij}(x)$ are nonvanishing on the support of $\bar{\phi}_j(x)$.
Our goal now is to show that each $b_j(\xi)$ satisfies the bounds of Theorem 1.3 if the $\epsilon$ determining the coefficients $c_i$ is sufficiently small.
In order for our arguments to work, we will need that for any given $j$ there is a single $l$ such that the $l$th component $\alpha_{ijl}$ is nonzero for all $i$
and such that $\alpha_{i_1j} > \alpha_{i_2j}$ implies $\alpha_{i_1jl} > \alpha_{i_2jl}$.
In order to ensure this is the case, we perform an additional (relatively simple) resolution of singularities. 

We proceed as follows. We divide $\R^n$
 (up to a set of measure zero) into $n!$ regions $\{A_k\}_{k=1}^{n!}$, where 
each $A_k$ is a region of the form $\{x \in \R^n: |x_{l_1}| < |x_{l_2}| < ... < |x_{l_n}|\}$, where  $x_{l_1},...,x_{l_n}$ is a permutation of
the $x$ variables. We write $(2.5')$ as $\sum_{k=1}^{n!} b_{jk}(\xi)$, where 
\[ b_{jk}(\xi) = \int_{A_k}\bar{\phi}_j(x)Jac_j(x)e^{-i|\xi|(c_1a_{1j}(x)m_{1j}(x) + ... + c_{n+1}a_{n+1\,j}(x)m_{n+1\,j}(x))}\,dx_1...\,dx_n \tag{2.6}\]
We focus our attention one one such $A_k$ and we let $u_i$ denote $x_{l_i}$. We make the variable changes $u_i = \prod_{p=i}^n y_p$. In the $y$ variables, $A_k$  becomes the rectangular box $(-1,1)^{n-1} \times \R$ and $(2.6)$ can be written in the form
\[ b_{jk}(\xi) = \int_{(-1,1)^{n-1} \times \R}\tilde{\phi}_{jk}(y)J_{jk}(y)e^{-i|\xi|(c_1b_{1jk}(y)p_{1jk}(y) + ... + c_{n+1}b_{n+1\,jk}(y)p_{n+1\,jk}(y))}\,dy_1...\,dy_n \tag{2.7}\]
Here $J_{jk}(y)$ is the Jacobian of the combined coordinate change, once again comparable to a monomial, the $p_{ijk}(y)$ are again monomials and the $b_{ijk}(y)$ are nonvanishing real analytic functions on the support of the smooth compactly
supported function
$\tilde{\phi}_{jk}(y)$. Like above, for fixed $j$ and $k$ the exponents of the monomials $p_{ijk}(y)$ are lexicographically ordered (not necessarily 
strictly) in some ordering. Furthermore, 
note that each $p_{ijk}(y)$ is of the form $y_1^{\beta_{ijk1}}...y_n^{\beta_{ijkn}}$ where $\beta_{ijkn}$ is the overall degree of the monomial 
$m_{ij}(x)$. In particular, each $\beta_{ijkn}$ is nonzero and $\beta_{i_1jk} > \beta_{i_2jk}$ implies $\beta_{i_1jkn} > \beta_{i_2jkn}$, 
achieving the goal of this second resolution of singularities. 

To perform the analysis $(2.7)$ we will use the Van der Corput lemma on $(2.7)$ in $y_n$ variable. If there are $i_1 > i_2$ such that 
$\beta_{i_1jk} \geq \beta_{i_2jk}$, then if the $\epsilon$ such that $|c_{i+1}| \leq \epsilon |c_i|$ for each $i$ is small enough, the term 
$ c_{i_1}b_{i_1jk}(y)p_{i_1jk}(y)$ will be negligible in comparison with $ c_{i_2}b_{i_2jk}(y)p_{i_2jk}(y)$ when performing such applications of the 
Van der Corput lemma. As a result the phase in $(2.7)$ will effectively have terms $c_ib_{ijk}(y)p_{ijk}(y)$ for which the $\beta_{ijk}$ are 
*strictly* decreasing in $i$. Furthermore,  the $\beta_{ijkn}$ too will be effectively strictly decreasing in $i$, as well as nonzero. The latter facts are needed
for our applications of the Van der Corput lemma to be successful.

Next, we show that we may assume that there is a constant $\delta_0 > 0$ to be determined by our arguments such that on the domain
 of integration of $(2.7)$, for each $i$ and each $1 \leq p \leq n+1$ we have
\[|\partial_{y_n}^p(b_{ijk}(y)p_{ijk}(y)) - \partial_{y_n}^p (b_{ijk}(y_1,...,y_{n-1},0)p_{ijk}(y))| \leq \delta_0 {1 \over |y_n|^p}|b_{ijk}(y)p_{ijk}(y)| \tag{2.8}\] 
Since the $b_{ijk}(y)$ are smooth functions and $p_{ijk}(y)$ is a monomial, this is necessarily true over the portion 
of the domain where $|y_n| < c$ for some $c > 0$ determined by $\delta_0$, and I claim that we can arrange that 
 such an inequality $|y_n| < c$ is automatically satisfied over the domain of integration of $(2.7)$.

To see why this is the case, first note that in terms of the $x$ variables of $(2.6)$, one will have that $|y_n| < c$ if one has  
$|x| < c'$ for some $c'$ depending on $c$. 
Next, let $x'$ be in the support of some $\alpha_j(\Psi_j(x))$ such that $\Psi_j(x') = 0$. Here $\alpha_j$ is as defined after $(2.4)$. There is a neighborhood of $x'$ on which we may shift 
coordinates to become centered at $x = x'$ instead of
$x = 0$; the functions that were monomialized before will be 
monomialized in the shifted coordinates, and the multiindices in the shifted coordinates will be lexicographically ordered in the shifted coordinates 
just as before.

Let $D_{x'}$ be a disk centered at $x'$ such that  the above considerations hold on $D_{x'}$ and such that 
each $|x|$ is less than the associated $c'$ on $D_{x'}$ in the coordinates centered at $x'$. By compactness, we
may let $W_j$ be a neighborhood of the points in the support of $\alpha_j(\Psi_j(x))$ with $\Psi_j(x) = 0$, that is a finite union of such $D_{x'}$. There is 
then an $r_{W_j} > 0$
such that if the original cutoff function $\phi(x)$ is supported on $|x| < r_{W_j}$, then the function 
$\phi_j(\Psi_j(x)) = \alpha_j(\Psi_j(x))\phi(\Psi_j(x))$, will be supported in $W_j$. (Otherwise, there would be a point outside of $W_j$ in the support 
of $\alpha_j(\Psi_j(x))$ getting mapped to the origin by $\Psi_j$).

We can then use a partition of unity subordinate to the $D_{x'}$ comprising $W_j$ to write 
$\phi_j(\Psi_j(x))$ as a finite sum $\sum_l{\phi}_{jl}(x)$ of smooth functions such that each ${\phi}_{jl}(x)$ is supported on one of the $D_{x'}$. 
We denote the $x'$ corresponding to ${\phi}_{jl}(x)$ by $x_{jl}'$.
If we use the ${\phi}_{jl}(x + x_{jl}')$ in place of $\bar{\phi}_j(x) = {\phi}_j(\Psi_j(x))$ and the ${\phi}_{jl} (\Psi_j^{-1}(x) - x_{jl}'))$ in place
of ${\phi}_j(x)$, then the associated $|y_n|$ will be as small as needed above. Thus so long as the original cutoff function $\phi(x)$ is supported on 
$\{x: |x| < \min_j r_{W_j}\}$,  we have arranged that for  each $j$ and $k$ the variable $|y_n|$ will be smaller than the $c$ needed for
 $(2.8)$ to hold, as desired.

We now proceed to the analysis of $(2.7)$. We consider the following variant of the phase function of $(2.7)$.
\[|\xi|\sum_{i \in K} c_i b_{ijk}(y_1,...,y_{n-1},0)p_{ijk}(y) \tag{2.9}\]
Here $K$ consists of those indices $i$ for which there is no $i'$ with $i' < i$ and $\beta_{i'jk} \leq \beta_{ijk}$ (equivalently, $\beta_{i'jkn} \leq \beta_{ijkn})$).
We view $(2.9)$ as a function of $y_n$ for fixed $(y_1,...,y_n)$. We suppress the
$y_1,...,y_{n-1}$ variables and the indices $j$ and $k$, writing $\sigma_i = \beta_{ijkn}$ and  letting $d_i = b_{ijk}(y_1,...,y_{n-1},0)y_1^{\beta_{ijk1}}...y_{n-1}^{\beta_{ijk\,n-1}}$, and write $(2.9)$ as  

\[h(y_n) = |\xi|\sum_{i \in K} c_i d_i y_n^{\sigma_i} \tag{2.10}\]

Write the $\sigma_i$ appearing in $(2.10)$ as $\sigma_{n_i}$ for $1 \leq i \leq |K|$. Since the $\sigma_i$ are distinct and nonzero, the vectors 
$\{(\sigma_{n_1}^p,...,\sigma_{n_{|K|}}^p): p = 1,...,|K|\}$ are linearly independent (as 
can be shown using the Vandermonde determinant). As a result, using row operations the vectors $w_p = (\prod_{q=0}^{p-1}(\sigma_{n_1} - q), ...,
\prod_{q=0}^{p-1} (\sigma_{n_{|K|}} - q) )$ are linearly independent for $p = 1,...,|K|$. As a result, there exists an $\eta > 0$ such that for
any vector $v \in \R^n$ there exists a $p$ with $1 \leq p \leq |K|$ such that
\[|w_p \cdot v| \geq \eta|v| \tag{2.11}\]

Letting $v =  c_{n_i} d_{n_i} y_n^{\sigma_{n_i}}$, equation $(2.11)$ implies that for each $y_n$ there necessarily exists a $p$ with $1 \leq p \leq |K|$ (which can depend on $y_n$) such that
\[\bigg|{d^p h  \over d y_n^p}(y_n)\bigg| > \eta{1 \over |y_n|^p} |\xi| \sum_{i \in K} |c_i d_i y_n^{\sigma_i} | \tag{2.12}\]
Since there is a constant $C$ such that $\bigg|{d^{p+1} h \over d y_n^{p+1}}(y_n)\bigg| < C{1 \over |y_n|^{p+1}} |\xi| \sum_{i \in K} |c_i d_i y_n^{\sigma_i} | $
for all $p$ with $1 \leq p \leq |K|$ and all $y_n$, in view of $(2.12)$ each dyadic interval in $y_n$ can be written as the union of boundedly many subintervals on each of which we have
for some $1 \leq p \leq |K|$ that
\[\bigg|{d^p h \over d y_n^p}(y_n)\bigg| > {\eta \over 2} {1 \over |y_n|^p} |\xi| \sum_{i \in K} |c_i d_i y_n^{\sigma_i} | \tag{2.13}\]
We would like $(2.13)$ to hold with $h(y_n)$ replaced by $|\xi| \sum_{i \in K} c_i b_{ijk}(y)p_{ijk}(y)$. That is, we want to be able to 
replace the $ b_{ijk}(y_1,...,y_{n-1},0)y_1^{\beta_{ijk1}}...y_{n-1}^{\beta_{ijk\,n-1}}$ factors in $h(t)$ by the function
$b_{ijk}(y)y_1^{\beta_{ijk1}}...y_{n-1}^{\beta_{ijk\,n-1}}$. This follows from $(2.8)$, which by the discussion after $(2.8)$ holds if  $|y_n| < c$,
which we showed we may assume. Hence letting $j(y) = |\xi|\sum_{i \in K} c_ib_{ijk}(y)p_{ijk}(y)$ we can assume that in place of $(2.13)$ we have
\[\bigg|{d^p j \over d y_n^p}(y_n)\bigg| > {\eta \over 4} {1 \over |y_n|^p} |\xi| \sum_{i \in K} |c_i d_i y_n^{\sigma_i} | \tag{2.14}\]
Next, we would like for $(2.14)$ to hold with $j(y)$ replaced by the phase function $P(y) = |\xi| \sum_{i=1}^{n+1} c_ib_{ijk}(y)p_{ijk}(y)$ of $(2.7)$, and with the sum on the right hand
side being over all $i$ and not just all $i \in K$. In other words, in both sides of the equation we would like to include the terms with $i \notin K$. 

If $i \notin K$, there must be an $i' < i$
such that $\beta_{i'jk} \leq \beta_{ijk}$. As a result, if the $\epsilon$ such that $|c_{j+1}| \leq \epsilon |c_j|$ for all $j$ is chosen sufficiently small, the term $|c_i d_iy_n^{\sigma_i}|$ will be
less than say ${\eta \over 100 n} |c_{i'} d_{i'}y_n^{\sigma_{i'}}|$. Thus including such terms will not affect the right-hand side of $(2.14)$ beyond a factor of 2. Similarly, if $\epsilon$ is 
sufficiently small then for any $i \notin K$ and any $1 \leq p \leq |K|$ the term ${\partial^p \over \partial y_n^p} (c_ib_{ijk}(y)y_1^{\beta_{ijk1}}...y_{n}^{\beta_{ijkn}})$ can be made smaller than
${\eta \over 100 n}{1 \over |y_n|^p} |c_{i'} d_{i'} y_n^{\sigma_{i'}}|$, assuming as we may that $|y_n| < c$ for a sufficiently small $c$. Hence including such terms will not affect $(2.14)$ beyond an additional factor of 2. Thus we may include these
new terms in both sides of $(2.14)$ and we obtain that on each $y_n$-interval in question there is a $p$ with $1 \leq p \leq  |K|$ such that
\[\bigg|{d^p \over d y_n^p}\bigg(|\xi| \sum_{i=1}^{n+1}c_i b_{ijk}(y)p_{ijk}(y)\bigg) \bigg| > {\eta \over 16} {1 \over |y_n|^p} |\xi| \sum_{i =1}^{n+1}|c_i d_i y_n^{\sigma_i} | \tag{2.15}\]
In the integral $(2.7)$, for fixed $y_1,...,y_{n-1}$, we now apply Lemma 2.1 or 2.2 in the $y_n$ directions on one of the boundedly many subintervals of some dyadic 
interval $2^{-q-1} \leq |y_n| < 2^{-q}$ on which $(2.15)$ holds. Denote this subinterval by $J$ and let $y_n'$ denote the center of this interval.
We obtain
\[\bigg|\int_{J}\tilde{\phi}_{jk}(y)J_{jk}(y)e^{-i|\xi|(c_1b_{1jk}(y)p_{1jk}(y) + ... + c_{n+1}b_{n+1\,jk}(y)p_{n+1\,jk}(y))}\,dy_n\bigg| \]
\[\leq C J_{jk}(y_1,...,y_{n-1},y_n')|y_n'| (|\xi|\sum_{i=1}^{n+1} |c_id_i(y_n')^{\sigma_i}|)^{-{1 \over p}} \tag{2.16}\]
Note that we use the fact that $J_{jk}(y)$ is comparable to a monomial when applying Lemma 2.1 or Lemma 2.2 here, since this implies that $J_{jk}(y)$
is comparable to the fixed value $J_{jk}(y_1,...,y_{n-1},y_n')$ on $J'$, and gives appropriate bounds on its derivatives. 

Because ${1 \over 2} \leq c_1 \leq 1$ and the other $c_i$ can be arbitrarily small, even
zero, it makes sense to only include the first term in the sum on the right hand side of $(2.16)$, and we obtain
\[\bigg|\int_{J}\tilde{\phi}_{jk}(y)J_{jk}(y)e^{-i|\xi|(c_1b_{1jk}(y)p_{1jk}(y) + ... + c_{n+1}b_{n+1\,jk}(y)p_{n+1\,jk}(y))}\,dy_n\bigg| \]
\[\leq C J_{jk}(y_1,...,y_{n-1},y_n')|y_n'| (|\xi| |d_1(y_n')^{\sigma_1}|)^{-{1 \over p}} \tag{2.17}\]
Since $y_n \sim y_n'$ on $J$,  if $J'$ denotes the dyadic interval containing $J$  this can be rewritten as 
\[\bigg|\int_{J}\tilde{\phi}_{jk}(y)J_{jk}(y)e^{-i|\xi|(c_1b_{1jk}(y)p_{1jk}(y) + ... + c_{n+1}b_{n+1\,jk}(y)p_{n+1\,jk}(y))}\,dy_n\bigg| \]
\[\leq C \int_{J'} J_{jk}(y)
(|\xi||d_1(y_n')^{\sigma_1}|)^{-{1 \over p}}\,dy_n \tag{2.18}\]
Since $\sigma_1 = \beta_{1jkn}$ and $d_1 = b_{1jk}(y_1,...,y_{n-1},0)y_1^{\beta_{1jk1}}...y_{n-1}^{\beta_{1jk\,n-1}}$ is comparable in magnitude on $J$ to 
$b_{1jk}(y_1,...,y_{n-1},y_n)y_1^{\beta_{1jk1}}...y_{n-1}^{\beta_{1jk\,n-1}}$, the factor $|d_1(y_n')^{\sigma_1}|$ on the right-hand side of $(2.18)$ is
comparable in magnitude to $| b_{1jk}(y)y^{\beta_{1jk}}| = | b_{1jk}(y)p_{ijk}(y)| $. Consequently, $(2.18)$ can be rewritten as
\[\bigg|\int_{J}\tilde{\phi}_{jk}(y)J_{jk}(y)e^{-i|\xi|(c_1b_{1jk}(y)p_{1jk}(y) + ... + c_{n+1}b_{n+1\,jk}(y)p_{n+1\,jk}(y))}\,dy_n\bigg| \]
\[\leq C \int_{J'} J_{jk}(y)
(|\xi||b_{1jk}(y)p_{1jk}(y)|)^{-{1 \over p}}\,dy_n \tag{2.19}\]
By simply taking absolute values and integrating, we also have
\[\bigg|\int_J \tilde{\phi}_{jk}(y)J_{jk}(y)e^{-i|\xi|(c_1b_{1jk}(y)p_{1jk}(y) + ... + c_{n+1}b_{n+1\,jk}(y)p_{n+1\,jk}(y))}\,dy\bigg| \leq C_1 \int_{J'} J_{jk}(y) \,dy \tag{2.20}\]
Combining $(2.19)$ and $(2.20)$ then gives
\[\bigg|\int_J \tilde{\phi}_{jk}(y)J_{jk}(y)e^{-i|\xi|(c_1b_{1jk}(y)p_{1jk}(y) + ... + c_{n+1}b_{n+1\,jk}(y)p_{n+1\,jk}(y))}\,dy\bigg|\]
\[\leq C_2 \int_{J' } J_{jk}(y) \min(1, (|\xi||b_{1jk}(y)p_{1jk}(y)|)^{-{1 \over p}})\,dy \tag{2.21}\]
Since $1 \leq p \leq n+1$, the right-hand side of $(2.21)$ is maximized for $p = n + 1$. Inserting $p = n + 1$ into $(2.21)$, then
adding $(2.21)$ over the boundedly many intervals $J$ comprising $J'$, and then integrating the result in 
the $y_1,...,y_{m-1}$ variables leads to the following.
\[\bigg|\int_{(-1,1)^{n-1} \times J'} \tilde{\phi}_{jk}(y)J_{jk}(y)e^{-i|\xi|(c_1b_{1jk}(y)p_{1jk}(y) + ... + c_{n+1}b_{n+1\,jk}(y)p_{n+1\,jk}(y))}\,dy\bigg|\]
\[\leq C \int_{(-1,1)^{n-1} \times J'} J_{jk}(y) \min(1, (|\xi||b_{1jk}(y)p_{1jk}(y)|)^{-{1 \over p}})\,dy \tag{2.22}\]
Adding $(2.22)$ over all  intervals $J'$, for some $\delta_0 > 0$ we get
the following bound for $|b_{jk}(\xi)|$ in $(2.7)$:
\[|b_{jk}(\xi)| \leq C_3\int_{(-1,1)^{n-1} \times (-\delta_0,\delta_0) } J_{jk}(y) \min(1, (|\xi||b_{1jk}(y)p_{1jk}(y)|)^{-{1 \over p}})\,dy \tag{2.23}\]
The form of $(2.23)$ is such that if we replace $\delta_0$ by some $\delta_1 < \delta_0$, then inequality $(2.23)$
  will still be valid, albeit with a different constant. So for any such $\delta_1$ we have
\[|b_{jk}(\xi)| \leq C_{\delta_1} \int_{(-1,1)^{n-1} \times (-\delta_1,\delta_1)} J_{jk}(y) \min(1, (|\xi||b_{1jk}(y)p_{1jk}(y)|)^{-{1 \over p}})\,dy \tag{2.24}\]
In particular, we may assume $\delta_1$ is small enough so that the pullback of $(-1,1)^{n-1} \times [-\delta_1,\delta_1]$ under the coordinate changes
of the above resolution of singularities is contained in a set $\{x: |x| < r\}$ on which one has an asymptotic development of the form $(1.1)$ for
$g_1(x)$. We convert $(2.24)$ back into the original $x$ coordinates through
these coordinate changes, turning $b_{1jk}(y)p_{1jk}(y)$ back into $g_1(x)$. We get that for some $C_4, r > 0$ we have
\[|b_{jk}(\xi)| \leq C_4\int_{\{x: |x| < r\}} \min(1,(|\xi||g_1(x)|)^{-{1 \over n + 1}})\,dx \tag{2.25}\]
Adding this over all $j$ and $k$ then gives the following bound on the original Fourier transform $\widehat{\phi(x)\mu}(\xi)$ that we are bounding for those
$\xi$ for which ${\xi \over |\xi|} \in A_{v_1,...,v_{n+1},\epsilon} = \{c_1v_1 + ... + c_{n+1}v_{n+1}: |c_1| \leq 1, |c_{i+1}| \leq \epsilon|c_i|$ for all $i \geq 1\}$.
\[|\widehat{\phi(x)\mu}(\xi)| \leq C_5\int_{\{x: |x| < r\}} \min(1,(|\xi||g_1(x)|)^{-{1 \over n + 1}})\,dx \tag{2.26}\]
Recall that $g_1(x)$ is of the form $v_1 \cdot (x_1,...,x_n, f(x_1,...,x_n))$, where $f(x)$ is as in $(2.3)$ and $|v_1| = 1$. Recall $f(0) = 0$ and 
$\nabla f(0) = 0$. So the growth rate as in $(1.1)$ of $v_1 \cdot (x_1,...,x_n, f(x_1,...,x_n))$ for $|v_1| = 1$ is given by $(g,k) = (1,0)$ unless 
$v_1$ is $(0,...,0, \pm 1)$, in which case $v_1 \cdot (x_1,...,x_n, f(x_1,...,x_n)) = \pm f(x_1,...,x_n)$ and the growth rate is slower. Thus the decay rate of
the  right hand side of $(2.26)$ is slowest  when $g_1(x) = f(x)$. Hence we have
\[|\widehat{\phi(x)\mu}(\xi)| \leq C_6 \int_{\{x: |x| < r\}} \min(1,(|\xi||f(x)|)^{-{1 \over n + 1}})\,dx \tag{2.27}\]
We now are in a position to prove parts 1 - 3 of Theorem 1.3. It is natural to break the right-hand side $(2.25)$ into $|f(x)| > {1 \over |\xi|}$ and
 $|f(x)| < {1 \over |\xi|}$ parts. We obtain
 \[|\widehat{\phi(x)\mu}(\xi)| \leq C_6 m(\{x: |x| < r, |f(x)| < |\xi|^{-1}\}) + C_6|\xi|^{-{1 \over n+1}}\int_{\{x: |x| < r,\, |f(x)| > |\xi|^{-1}\}}|f(x)|^{-{1 \over n + 1}}\,dx \tag{2.28}\]
 If $(g_0,k_0)$ denotes the growth rate of $f(x)$ as in $(1.1)$, then the above gives for $|\xi| > 2$ that
 \[|\widehat{\phi(x)\mu}(\xi)| \leq C_7|\xi|^{-g_0}(\ln |\xi|)^{k_0} + C_6|\xi|^{-{1 \over n+1}}\int_{\{x: |x| < r,\,|f(x)| > |\xi|^{-1}\}}|f(x)|^{-{1 \over n + 1}}\,dx \tag{2.29}\]
We apply to $|f(x)|^{-1}$ the  characterization of integrals of powers of functions in terms of their distribution functions. Then
the integral in $(2.29)$ becomes
\[{1 \over n + 1}\int_{|\xi|^{-1}}^{\infty}t^{-{{1 \over n + 1} - 1}}m(x: |x| < r,\, |\xi|^{-1}
< |f(x)| < t\})\,dt \tag{2.30a}\]
\[ \leq {1 \over n + 1}\int_{|\xi|^{-1}}^{{1 \over 2}}t^{-{{1 \over n + 1} - 1}}m(x: |x| < r,\, |\xi|^{-1}
< |f(x)| < t\})\,dt + C_8 \tag{2.30b}\]
Thus by $(1.1)$, we have that $(2.30b)$ is bounded by
\[C_9 \int_{|\xi|^{-1}}^{{1 \over 2}}t^{-{{1 \over n + 1} + g_0 - 1}}|\ln t|^{k_0} \,dt  + C_8\tag{2.31}\]
Given $(2.31)$ and the fact that first term in $(2.29)$ is bounded by $C_7|\xi|^{-g_0}(\ln |\xi|)^{k_0}$, we conclude that
\[|\widehat{\phi(x)\mu}(\xi)| \leq C_{10}\bigg(|\xi|^{-g_0}(\ln |\xi|)^{k_0} + |\xi|^{-{1 \over n+1}} +  |\xi|^{-{1 \over n + 1}} \int_{|\xi|^{-1}}^{{1 \over 2}}t^{-{{1 \over n + 1} + g_0 - 1}}|\ln t|^{k_0}) \,dt \bigg) \tag{2.32}\]
If $g_0 < {1 \over n + 1}$, equation $(2.32)$ leads to
$|\widehat{\phi(x)\mu}(\xi)|  \leq C_{11}|\xi|^{-g_0}(\ln |\xi|)^{k_0}$, which is the desired estimate for part 1 of Theorem 1.3. If $g_0 = {1 \over n + 1}$
we gain an additional logarithmic factor, giving part 2. If $g_0 > {1 \over n + 1}$, we get that
$|\widehat{\phi(x)\mu}(\xi)|$ is bounded a constant times $|\xi|^{-{1 \over n + 1}}$, the desired estimate for part 3 of Theorem 1.3. This concludes the 
proof of parts 1 through 3 of Theorem 1.3. By the discussion above the statement of Theorem 1.2, parts 1 through 3 of Theorem 1.2 then follow. \qed

\section {The proof of parts 4 of Theorem 1.2 and 1.3.}

We first prove part 4 of Theorem 1.3. Rotating and translating coordinates if necessary, we may assume without loss of generality that $x_0 = 0$
 and the vector $(0,...,0,1)$ is normal to $\partial D$ at $0$. In the setting of part 4 of Theorem 1.3, $g_{x_0} = g_{0} < 1$ and if the support of $\phi$ is 
 sufficiently small, which we may assume, we have
\[\widehat{\phi(x)\mu}(\xi) = \int_{\R^n}\phi_0(x_1,...,x_n) e^{-i\xi_{n+1}f(x_1,...,x_n) - i\xi_1x_1 -... -i\xi_nx_n}\,dx_1...\,dx_n \tag{3.1}\]
Here $\partial D$ is the graph of $f(x)$ and $\phi_0(x)$ is a smooth, compactly supported, nonnegative, and positive function on a neighborhood of the origin. 
To prove part 4 of Theorem 1.3, we let $g'$ satisfy $g_0 < g' < 1$ and suppose that the estimate $|\widehat{\phi(x)\mu}(\xi) | \leq C|\xi|^{g'}$
holds for sufficiently large $|\xi|$. We will arrive at a contradiction. 

\noindent We examine $(3.1)$ in the $(0,...,0,1)$ direction:
\[ \bigg|\int_{\R^n}\phi_0(x)e^{-i\xi_{n+1}f(x)}\,dx\bigg| \leq C (1 + |\xi_{n+1}|)^{-g'} \tag{3.1'}\]

The following is largely taken from the proof of part 2 of Theorem 1.1 of [G4]. Denote the integral on the left of $(3.1')$ by $I(\xi_{n+1})$.
Let $B(x)$ be a smooth function on $\R$  whose Fourier transform is nonnegative, compactly supported, and equal to 1 on a 
neighborhood of the origin, and let $\epsilon$ be a small  positive number. If $0 <  g'' < g'$, then $(3.1')$ implies
 that for some constant $A$ independent of $\epsilon$ one has
\[\int_{\R}|I(\xi_{n+1}) \xi_{n+1}^{g''  - 1}B(\epsilon \xi_{n+1})|\,d \xi_{n+1} < A \tag{3.2}\]
As a result we have
\[\bigg|\int_{\R^{n+1}} e^{-i\xi_{n+1}f(x)} \phi_0(x)|\xi_{n+1}|^{g''  - 1}B(\epsilon \xi_{n+1})\,d\xi_{n+1} \,dx\bigg|  <  A \tag{3.3}\]
We do the integral in $\xi_{n+1}$ in $(3.3)$. Letting $b_{\epsilon}(y)$ be the convolution of $|y|^{-g''}$ with ${1 \over \epsilon} \hat{B}({y \over \epsilon})$,  for a constant $A'$ independent of $\epsilon$ we get
\[\bigg|\int_{\R^n}b_{\epsilon} (f(x)) \,\phi_0(x)\,dx\bigg| < A' \tag{3.4}\]
Note that both $b_{\epsilon} (f(x))$ and $\phi_0(x)$ are nonnegative here. Thus we may remove the absolute value and let
 $\epsilon \rightarrow 0$ to obtain
\[\int_{\R^n}|f(x)|^{-g''} \phi_0(x)< \infty  \tag{3.5}\]
Since $ \phi_0(x)$ is bounded below by $C > 0$ on a neighborhood $N$ of the origin, we therefore have
\[\int_{N}|f(x)|^{-g''}\,dx< \infty  \tag{3.6}\]
In other words, $|f(x)|^{-g''}$ is in $L^1(N)$, and therefore in weak $L^1$, and we have the existence of a constant $C$ such that
\[\mu (\{x \in N : |f(x)|^{-g''} > \epsilon \}) \leq C  {1 \over  \epsilon} \tag{3.7}\]
Replacing $ \epsilon$ by $ \epsilon^{-g''}$, we get
\[\mu (\{x \in N: |f(x)| <  \epsilon \}) \leq C  \epsilon^{g''} \tag{3.8}\]
However $0 < g'' < g'$ is arbitary, so since $g' > g_{x_0}$ we may select $g'' > g_{x_0}$. This contradicts that $(1.1)$ holds for all sufficiently small $r > 0$.
 Thus we have arrived at a contradiction and the proof of part 4 of Theorem 1.3 is complete.

Moving now to the proof of part 4 of Theorem 1.2, suppose $g < 1$ and there exists some $g' > g$ such that the estimate $|\widehat{\chi_D(x)}(\xi)| \leq C|\xi|^{-1 - g'}$ holds for sufficiently large $|\xi|$. We may take $g' < 1$. Let $x_0 \in \partial D$ such that $g_{x_0} < g'$. Like above,
 rotating and translating coordinates if necessary, we may assume that $x_0 = 0$ and  $(0,...,0,1)$ is normal to $\partial D$ at $x_0$. Let 
 $\psi(x)$ be a nonnegative
cutoff function defined on a neighborhood of $0$ for which $(1.12)$ is valid, and such that $\psi(x) > 0$ on a neighborhood of $0$. Then the estimate $|\widehat{\chi_D}(\xi)| \leq C|\xi|^{-1 - g'}$ implies, by looking at the Fourier transform of the product as a convolution, that $|\widehat{\chi_D(x) \psi(x)}(\xi)| 
\leq C|\xi|^{-1 - g'}$ also holds. In particular, it holds when $\xi$ is in the $(0,...,0,1)$ direction.

Examining $(1.12)-(1.14)$ in this direction, we observe that the  last two terms in $(1.14)$ decay at a rate of at least $C|\xi|^{-2}$, 
so $|\widehat{\chi_D(x) \psi(x)}(\xi)| \leq C|\xi|^{-1 - g'}$ implies that 
 the first term of $(1.14)$ must be bounded by $C'|\xi|^{-1 - g'}$. Thus the absolute value of the integral in the first term is bounded by $C'|\xi|^{-g'}$
 in this direction. But this integral is exactly a surface measure Fourier transform of the form $(3.1)$, with $\phi_0(x)$ replaced by $\psi_1(x)$. So 
the argument above shows that the bound of $C'|\xi|^{-g'}$ leads to a contradiction. This completes the proof of Theorem 1.2.
\qed
\section{The Proof of Theorem 1.4.}

\subsection {The proof of part 1 of Theorem 1.4.} 

We assume we are in the setting of part 1 of Theorem 1.4. So the index $h$ is equal to ${1 \over m + \gamma}$ for a nonnegative integer $m$ and some
$0 < \gamma \leq 1$, and $s_1(x),...,s_m(x)$ are linearly independent real analytic functions on a neighborhood of the origin in $\R^n$ with each 
$s_i(0) = 0$. Let $s(x)$ denote $(s_1(x),...,s_m(x))$. In order
to prove part 1 of Theorem 1.4, it suffices to show there exist positive constants $r_1, \delta$ and $C$ such that for each $v \in \R^m$ with 
$|v| < \delta$, one has
\[ m(\{x: |x| < r_1,\,|q(x) + (s(x) \cdot v)| < \epsilon\})  \leq C \epsilon^h |\ln \epsilon|^l \tag{4.1}\]
For mutually orthogonal unit vectors $v_1,...,v_m$, write ${\bf v} = (v_1,...,v_m)$. Suppose for each such ${\bf v}$ there is an 
 $\zeta_{\bf v}> 0$ such that there exists a set $D_{v_1,...,v_m,\zeta_{\bf v}} = \{v: v = c_1v_1 + ... + c_m v_m$ with $|c_1| \leq \zeta_{\bf v}$, 
$|c_{i+1}| \leq \zeta_{\bf v}|c_i|$ for all $i\}$ such that $(4.1)$ holds for all $v \in D_{v_1,...,v_m,\zeta_{\bf v}}$ for some $r_1 > 0$. Then by applying Lemma 2.4 to the dilates
${1 \over \zeta_{\bf v}}D_{v_1,...,v_m,\zeta_{\bf v}}$, there is a $\delta > 0$ such that the $v$ with  $|v| < \delta$ is a subset of a finite union of such
$D_{v_1,...,v_m,\zeta_{\bf v}}$. So in order to prove part 1 of Theorem 1.4 we may assume that $v_1,...,v_m$
are fixed and we are seeking such a set  $D_{v_1,...,v_m,\zeta_{\bf v}}$. Letting $t_i(x)  = s(x) \cdot v_i $, the goal is then to show that there exist some 
$C, \zeta_{\bf v}, r_1 > 0$  such that if $|c_1| \leq \zeta_{\bf v}$ and $|c_{i+1}| \leq \zeta_{\bf v}|c_i|$ for all $i$ then we have
\[ m(\{x: |x| < r_1,\,|q(x) + c_1 t_1(x) + ... + c_m t_m(x)| < \epsilon\})  \leq C \epsilon^h |\ln \epsilon|^l \tag{4.2}\]

\noindent Rewrite the left-hand side of $(4.2)$ as
\[\int_{\{x: |x| < r_1\}} \chi_{\{x: |q(x) + c_1 t_1(x) + ... + c_m t_m(x)| < \epsilon\}}(x)\,dx \tag{4.3}\]
We now argue analogously to after $(2.4)$, with the functions $g_1(x),...,g_{n+1}(x)$ replaced by $q(x), t_1(x),...,t_m(x)$. 
In analogy with $(2.7)$ we are led to bounding finitely many terms of the following form.
\[ \beta_{jk}(\epsilon) = \int_{(-1,1)^{n-1} \times \R}\tau_{jk}(y)J_{jk}(y) \chi_{\{y: |B_{1jk}(y)P_{1jk}(y) + c_1B_{2jk}(y)P_{2jk}(y)   ... + c_mB_{m+1jk}(y)P_{m+1jk}(y) | < \epsilon\}}(y)\,dy \tag{4.4}\]
Here $\tau_{jk}(y)$ is a bounded nonnegative function with compact support, $J_{jk}(y)$ is the Jacobian of the coordinate change into blown up coordinates, $B_{1jk}(y)P_{1jk}(y)$ is 
the function $q(x)$ in blown up coordinates, and for $i > 1$,  $B_{ijk}(y)P_{ijk}(y)$ is $t_{i-1}(x)$ in the blown up coordinates. Analogous to in section 2,
the $P_{ijk}(y)$ are monomials whose exponent multiindices are lexicographically ordered in some ordering, and the $B_{ijk}(y)$ are real analytic functions whose absolute values are bounded above and below by a  constant. 

Like in section 2, we integrate in $y_n$ first, breaking up each dyadic interval into boundedly many parts on which one may apply the measure
version of the Van der Corput lemma, Lemma 2.3. There is a slight variation here, in that instead of using lower bounds on some  
$p$th derivative for $1 \leq p \leq m+1$ in conjunction with Lemma 2.1 or Lemma 2.2, we use lower bounds on a $p$th derivative with $0 \leq p \leq m$
in conjunction with Lemma 2.3. Specifically, in place of $(2.15)$, 
we use that if the $\zeta$ defining the $c_i$ is small enough, then for each fixed $y_1,...,y_{n-1}$ each dyadic $y_n$ interval can be written as the finite
union of subintervals, uniformly bounded in number, on each of which for some $0 \leq p \leq m$ we have
\[\bigg|{d^p \over d y_n^p}\bigg(B_{1jk}(y)P_{1jk}(y) + c_1B_{2jk}(y)P_{2jk}(y)   ... + c_mB_{m+1jk}(y)P_{m+1jk}(y) \bigg) \bigg| \]
\[> {\eta \over 16} {1 \over |y_n|^p} (|B_{1jk}(y)P_{1jk}(y)|+ |c_1B_{2jk}(y)P_{2jk}(y)| +   ... + |c_mB_{m+1jk}(y)P_{m+1jk}(y)|)\tag{4.5}\]
Since the $c_i$ can be arbitrarily small, even zero, we will only use the first term on the right-hand side of $(4.5)$. So writing $\eta_0 = {\eta \over 16}$ what we will be using is
\[\bigg|{d^p \over d y_n^p}\bigg(B_{1jk}(y)P_{1jk}(y) + c_1B_{2jk}(y)P_{2jk}(y)   ... + c_mB_{m+1jk}(y)P_{m+1jk}(y) \bigg) \bigg| > \eta_0 {1 \over |y_n|^p} |B_{1jk}(y)P_{1jk}(y)| \tag{4.6}\]
Let $I$ denote a $y_n$ interval on which $(4.6)$ holds. Let $y_n'$ be the center of $I$. Then we have that  ${1 \over |y_n|^p} |B_{1jk}(y)P_{1jk}(y)|$ $ \sim $
${1 \over |y_n'|^p} |B_{1jk}(y')P_{1jk}(y')|$ on $I$. So if $p > 0$, using $(4.6)$, Lemma 2.3 says that on $I$ we have
\[ m(\{y_n:  |B_{1jk}(y)P_{1jk}(y) + c_1B_{2jk}(y)P_{2jk}(y)   ... + c_mB_{m+1jk}(y)P_{m+1jk}(y) | < \epsilon\}) \]
\[< C_1\epsilon^{1 \over p}
|y_n'| |B_{1jk}(y')P_{1jk}(y')|^{-{1 \over p}} \tag{4.7}\]
So if $p > 0$, in view of $(4.7)$ and using that ${1 \over |y_n|^p} |B_{1jk}(y)P_{1jk}(y)| \sim {1 \over |y_n'|^p} |B_{1jk}(y')P_{1jk}(y')|$, that $\tau_{jk}(y)$ is bounded, 
and that $J_{jk}(y)$ is comparable to a fixed value on any dyadic interval since it is comparable to a monomial, if $I'$ denotes the $y_n$ dyadic interval containing $I$ we have
\[\int_I \tau_{jk}(y)J_{jk}(y) \chi_{\{y: |B_{1jk}(y)P_{1jk}(y) + c_1B_{2jk}(y)P_{2jk}(y)   ... + c_mB_{m+1jk}(y)P_{m+1jk}(y) | < \epsilon\}}(y)\,dy_n\]
\[\leq  C_2 \int_{I'}J_{jk}(y)\epsilon^{1 \over p}|B_{1jk}(y)P_{1jk}(y)|^{-{1 \over p}} \,dy_n\tag{4.8}\]
Simply taking absolute values in the left-hand side of $(4.8)$ and integrating leads to
\[\int_I \tau_{jk}(y)J_{jk}(y) \chi_{\{y: |B_{1jk}(y)P_{1jk}(y) + c_1B_{2jk}(y)P_{2jk}(y)   ... + c_mB_{m+1jk}(y)P_{m+1jk}(y) | < \epsilon\}}(y)\,dy_n\]
\[\leq  C_3 \int_{I'} J_{jk}(y)\,dy_n \tag{4.9}\]
Combining $(4.8)$ and $(4.9)$ results in 
\[\int_I \tau_{jk}(y)J_{jk}(y) \chi_{\{y: |B_{1jk}(y)P_{1jk}(y) + c_1B_{2jk}(y)P_{2jk}(y)   ... + c_mB_{m+1jk}(y)P_{m+1jk}(y) | < \epsilon\}}(y)\,dy_n\]
\[\leq  C_4 \int_{I'}J_{jk}(y)\min(1,\epsilon^{1 \over p}|B_{1jk}(y)P_{1jk}(y)|^{-{1 \over p}}) \,dy_n\tag{4.10}\]
The above holds when $1 \leq p \leq m$, and the right-hand integrand is largest when $p = m$, so in such situations $(4.10)$ implies
\[\int_I \tau_{jk}(y)J_{jk}(y) \chi_{\{y: |B_{1jk}(y)P_{1jk}(y) + c_1B_{2jk}(y)P_{2jk}(y)   ... + c_mB_{m+1jk}(y)P_{m+1jk}(y) | < \epsilon\}}(y)\,dy_n\]
\[\leq  C_4 \int_{I'}J_{jk}(y)\min(1,\epsilon^{1 \over m}|B_{1jk}(y)P_{1jk}(y)|^{-{1 \over m}}) \,dy_n\tag{4.11}\]
Although the above assumed $p > 0$, $(4.11)$ will still hold if $p = 0$ in $(4.6)$. For if $(4.6)$ holds with $p = 0$, then if $|B_{1jk}(y)P_{1jk}(y) + c_1B_{2jk}(y)P_{2jk}(y)   ... + c_mB_{m+1jk}(y)P_{m+1jk}(y) | < \epsilon$ at any point in $I$, then by $(4.6)$ at that point one has $|B_{1jk}(y)P_{1jk}(y)| < {1 \over \eta_0}\epsilon$. As a result,
$\epsilon^{1 \over m}|B_{1jk}(y)P_{1jk}(y)|^{-{1 \over m}} > \eta^{1 \over m}$ at this point, in which case $\epsilon^{1 \over m}|B_{1jk}(y)P_{1jk}(y)|^{-{1 \over m}} > C\eta^{1 \over m}$ over all of $I'$ for some $C$. Hence $\min(1,\epsilon^{1 \over m}|B_{1jk}(y)P_{1jk}(y)|^{-{1 \over m}})$ is
bounded by a constant on $I'$ and $(4.11)$ holds simply by taking absolute values inside the integral and integrating. 

So we assume that $(4.11)$ holds regardless of what $p$ is. Adding this over the boundedly many intervals $I$ comprising $I'$, $(4.11)$ leads to
\[\int_{I'} \tau_{jk}(y)J_{jk}(y) \chi_{\{y: |B_{1jk}(y)P_{1jk}(y) + c_1B_{2jk}(y)P_{2jk}(y)   ... + c_mB_{m+1jk}(y)P_{m+1jk}(y) | < \epsilon\}}(y)\,dy_n\]
\[\leq  C_5 \int_{I'}J_{jk}(y)\min(1,\epsilon^{1 \over m}|B_{1jk}(y)P_{1jk}(y)|^{-{1 \over m}}) \,dy_n\tag{4.12}\]
We then add $(4.12)$ over all dyadic intervals $I'$ and then integrate the result in the $y_1,...,y_{n-1}$ variables. The result is that for some $\delta_0 > 0$
we have
\[\int_{[-1,1]^{n-1} \times \R} \tau_{jk}(y)J_{jk}(y) \chi_{\{y: |B_{1jk}(y)P_{1jk}(y) + c_1B_{2jk}(y)P_{2jk}(y)   ... + c_mB_{m+1jk}(y)P_{m+1jk}(y) | < \epsilon\}}(y)\,dy\]
\[\leq  C_5 \int_{[-1,1]^{n-1} \times [-\delta_0,\delta_0]}J_{jk}(y)\min(1,\epsilon^{1 \over m}|B_{1jk}(y)P_{1jk}(y)|^{-{1 \over m}}) \,dy\tag{4.13}\]
Note that the left hand side of $(4.13)$ is the function $\beta_{jk}(\epsilon)$ of $(4.4)$, and recall that in order to prove part 1 of Theorem 1.4 we must 
show that each $\beta_{jk}(\epsilon)$ is bounded by the right hand side of $(1.18)$, namely $C \epsilon^h |\ln \epsilon|^l$, where $(h,l)$ is as in $(1.16)$.
Equation $(4.13)$ is analogous to $(2.22)$, and we can argue as we did after $(2.22)$. The result is the following analogue of $(2.26)$.
\[\beta_{jk}(\epsilon) \leq C_6 \int_{\{x: |x| < r_1\}}\min(1, \epsilon^{1 \over m}|q(x)|^{-{1 \over m}})\,dx \tag{4.14}\]
Then the steps from $(2.26)$ onwards lead to a bound $C \epsilon^h |\ln \epsilon|^l$ so long as $h < {1 \over m}$, which holds by the definition of $h$ as 
${1 \over m + \gamma}$ with $0 < \gamma \leq 1$. Thus we have the needed bound and part 1 of Theorem 1.4 follows.

\subsection{The proof of part 2 of Theorem 1.4}

We consider the function $f(x_1,...,x_n) = x_1^{2k} + ... + x_n^{2k}$ where $k$ is a positive integer.
An easy computation reveals that for $f(x)$, the pair $(h,l)$ is given by $({n \over 2k},0)$. Let $t$ be any positive rational number and write $t = {a \over b}$ 
where $a$ and $b$ are positive integers. Let $u = {ca \over b}$ where $c$ is the smallest positive integer such that $u  > 1$. Then
since $u = {2ca \over 2b}$, we can let $n = 2ca$ and $k = b$ and we have that $f(x_1,...,x_n) = x_1^{2k} + ... + x_n^{2k}$ has as its $(h,l)$ the pair
$(u,0)$. 

Next, note that for any $\delta > 0$, however small, $f(x_1,...,x_n) + \delta x_1$ has $(1,0)$ as its $(h,l)$. Since $u > 1$, this means that the pair
$(h,l)$ corresponding to $f(x_1,...,x_n)$ gets worse for any such perturbation. Now let $q(x) = (f(x))^c$. Then the pair $(h,k)$ corresponding to
$q(x)$ is $({u \over c}, 0) = ({a \over b},0) = (t,0)$, and for any $\delta > 0$ the corresponding pair for $(f(x) + \delta x_1)^c$ is $(1/c,1)$, which is worse
since $u > 1$.
Hence one does not have uniformity in the sense of Theorem 1.4 part 1 for $q(x)$ where the space of perturbation functions is the $c$ dimensional
space spanned by the $f(x)^k x_1^{c-k}$ for $0 \leq k \leq c - 1$. Hence for the function $q(x)$, the conclusion of part 1 of Theorem 1.4 does not hold
for $c$ dimensional perturbation families, and in fact one has the stronger fact that the index $h$ gets worse for these perturbations.

Since $c$ is defined to be the smallest integer such that $ct > 1$, if one writes $t = {1 \over m + \gamma}$ where $0 \leq \gamma < 1$ as in the 
statment of part 2 of Theorem 1.4, then $c = m + 1$. Therefore the statement of part 2 of Theorem 1.4 follows and we are done. \qed

\section{References}

\noindent [AGuV] V. Arnold, S. Gusein-Zade, A Varchenko, {\it Singularities of differentiable maps
Volume II}, Birkhauser, Basel, 1988.  \setlength{\parskip}{0.3 em}

\noindent [BoIw] E. Bombieri, H. Iwaniec, {\it On the order of $\zeta({1 \over 2} + it)$}, Ann. Scuola Norm. Sup. Pisa. Cl. Sci. (4) {\bf 13} (1986) 449-472.

\noindent [B] J. Bourgain, {\it Averages in the plane over convex curves and maximal operators},
J. Anal. Math. {\bf 47} (1986), 69-85. 

\noindent [BWa] J. Bourgain, N. Watt, {\it Mean square of zeta function, circle problem and divisor problem revisited}, preprint. arxiv 1709.04340.

\noindent [BraHoI] L. Brandolini, S. Hoffmann, A. Iosevich, {\it Sharp rate of average decay of the Fourier transform of a bounded set},
Geom. Funct. Anal. {\bf 13} (2003), no. 4, 671-680. 

\noindent [BrNaW] J. Bruna, A. Nagel, and S. Wainger, {\it Convex hypersurfaces and Fourier transforms},
Ann. of Math. (2) {\bf 127} no. 2, (1988), 333-365.  

\noindent [BuDeIkM] S. Buschenhenke, S. Dendrinos, I. Ikromov,  D. M\"uller, {\it Estimates for maximal functions associated to hypersurfaces in
 $\R^3$ with height $h<2$: Part I}, Trans. Amer. Math. Soc. {\bf 372} (2019), no. 2, 1363-1406.

\noindent [C] M. Christ, {\it Hilbert transforms along curves. I. Nilpotent groups}, Annals of Mathematics (2) {\bf 122} (1985), no.3, 575-596.

\noindent [CoDiMaM] M. Cowling, S. Disney, G. Mauceri, and D. M\"uller {\it Damping oscillatory integrals}, 
Invent. Math. {\bf 101}  (1990),  no. 2, 237-260.

\noindent [CoMa1] M. Cowling and G. Mauceri, {\it Inequalities for some maximal functions. II}, Trans. Amer. Math.
Soc., {\bf 296} (1986), no. 1, 341-365.

\noindent [CoMa2] M. Cowling, G. Mauceri, {\it Oscillatory integrals and Fourier transforms of surface carried measures},
Trans. Amer. Math. Soc.  {\bf 304} (1987), no. 1, 53-68.

\noindent [DK] J. DeMailly, J. Kollar, {\it Semicontinuity of complex singularity exponents and K\"ahler-Einstein metrics on Fano orbifolds},
 Ann. Sci. Ecole Norm. Sup. (4) {\bf 34} (2001), no. 4, 525-556.

\noindent [G1] M. Greenblatt, {\it Resolution of singularities, asymptotic expansions of oscillatory 
integrals, and related Phenomena}, J. Analyse Math. {\bf 111} no. 1 (2010) 221-245.

\noindent [G2] M. Greenblatt, {\it Oscillatory integral decay, sublevel set growth, and the Newton 
polyhedron}, Math. Annalen {\bf 346} (2010) no. 4, 857-890. 

\noindent [G3] M. Greenblatt, {\it Fourier transforms of irregular mixed homogeneous hypersurface measures},  Math. Nachr. {\bf 291} (2018), no. 7, 1075-1087. 

\noindent [G4] M. Greenblatt, {\it $L^p$ Sobolev regularity of averaging operators over hypersurfaces and the Newton polyhedron},  J. Funct. Anal. {\bf 276} (2019), no. 5, 1510-1527. 

\noindent [G5] M. Greenblatt, {\it Maximal averages over hypersurfaces and the Newton polyhedron}, J. Funct. Anal. {\bf 262} (2012), no. 5, 2314-2348. 

\noindent [G6] M. Greenblatt, {\it $L^p$ Sobolev regularity for a class of Radon and Radon-like transforms of various codimension},
 J. Fourier Anal. Appl. {\bf 25} (2019), no. 4, 1987-2003. 

\noindent [G7] M. Greenblatt, {\it Smooth and singular maximal averages over 2D hypersurfaces and associated Radon transforms},
 Adv. Math. {\bf 377} (2021), 107-465.
 
\noindent [Gr] A. Greenleaf, {\it Principal curvature and harmonic analysis}, Indiana Univ. Math. J. {\bf 30} (1981), no. 4, 519-537.

\noindent [Ha1] G. H. Hardy, {\it On the expression of a number as the sum of two squares},  Quart. J. Math. {\bf 46} (1915), 263-283.

\noindent [Ha2] G. H. Hardy, {\it On Dirichllet's divisor problem}, Proc. London Math. Soc. (2) {\bf 15} (1916) 1-25.

\noindent [He] D.R. Heath-Brown, {\it Lattice points in the sphere}, Number theory in progress, Proc. Number
Theory Conf. Zakopane 1997, eds. K. Gy¨ory et al., {\bf 2} (1999), 883-892.

\noindent [H1] H. Hironaka, {\it Resolution of singularities of an algebraic variety over a field of characteristic zero I}, 
 Ann. of Math. (2) {\bf 79} (1964), 109-203.

\noindent [H2] H. Hironaka, {\it Resolution of singularities of an algebraic variety over a field of characteristic zero II},  
Ann. of Math. (2) {\bf 79} (1964), 205-326. 

\noindent [Hl1] E. Hlawka, {\it \"{U}ber Integrale auf konvexen K\"{o}rpern. I.}  Monatsh. Math. {\bf 54} (1950), 1-36. 

\noindent [Hl2] E. Hlawka, {\it \"{U}ber Integrale auf konvexen K\"{o}rpern. II.} Monatsh. Math. {\bf 54} (1950), 81-99. 

\noindent [Hu1] M. N. Huxley, {\it Integer points, exponential sums and the Riemann zeta function}, Number theory for the millennium, II (Urbana, IL, 2000) pp.275-290, A K Peters, Natick, MA, 2002.

\noindent [Hu2] M. N. Huxley, {\it Area, lattice points, and exponential sums},
London Mathematical Society Monographs. New Series, 13. Oxford Science Publications. The Clarendon Press, Oxford University Press, New York, 1996. xii+494 pp. ISBN: 0-19-853466-3.

\noindent [Hu3] M. N. Huxley, {\it Exponential sums and lattice points III}, Proc. London Math. Soc. (3) {\bf 97} (2003),
591-609.

\noindent [I] A. Iosevich, {\it Maximal operators associated to families of flat curves in the plane}, Duke Math. J. {\bf 76} no. 2 (1994) 633-644. 

\noindent [IvKrKuNo] A. Ivi\'c, E. Kr\"{a}tzel, M. K\"{u}hleitner, W.G. Nowak, 
{\it Lattice points in large regions and related arithmetic functions: recent developments in a very classic topic}, Elementare und analytische Zahlentheorie,
Schr. Wiss. Ges. Johann Wolfgang Goethe Univ. Frankfurt am Main {\bf 20} (2006)  89-128.

\noindent [IkKeM] I. Ikromov, M. Kempe, and D. M\"uller, {\it Estimates for maximal functions associated
to hypersurfaces in $\R^3$ and related problems of harmonic analysis}, Acta Math. {\bf 204} (2010), no. 2,
151--271.

\noindent [ISa1] A. Iosevich, E. Sawyer, {\it Maximal averages over surfaces},  Adv. Math. {\bf 132} 
(1997), no. 1, 46-119.

\noindent [ISa2] A. Iosevich, E. Sawyer, {\it Oscillatory integrals and maximal averages over homogeneous
surfaces}, Duke Math. J. {\bf 82} no. 1 (1996), 103-141.

\noindent [ISaSe1] A. Iosevich, E. Sawyer, A. Seeger, {\it Two problems associated with convex
finite type domains}, Publ. Mat. {\bf 46}, no. 1 (2002), 153-177.

\noindent [ISaSe2] A. Iosevich, E. Sawyer, A. Seeger, {\it Mean square discrepancy bounds for the number of lattice points in large convex bodies}, J. Anal. Math. {\bf 87} (2002), 209-230.

\noindent [ISaSe3] A. Iosevich, E. Sawyer, A. Seeger, {\it Mean lattice point discrepancy bounds. II. Convex domains in the plane}, J. Anal. Math. 
{\bf 101} (2007), 25-63. 

\noindent [ISaSe4] A. Iosevich, E. Sawyer, A. Seeger, {\it On averaging operators associated with convex hypersurfaces of finite type}, J. Anal. Math. 
{\bf 79} (1999), 159-187.

\noindent [IwMo] H. Iwaniec and C.J. Mozzochi, {\it On the divisor and circle problems}, J. Number Theory {\bf 29} (1988),
60-93.

\noindent [J1] V. Jarnik, {\it Sur les points \`a coordon\'ees entieres dans le plan},  Bull. int. de l’acad. des sciences de Boh., 1924.

\noindent [J2] V. Jarnik, {\it \"{U}ber Gitterpunkte auf konvexen Kurven}, Math. Z. {\bf 24} (1925), 500-518.

\noindent [Ka1] V. N. Karpushkin, {\it A theorem concerning uniform estimates of oscillatory integrals when
the phase is a function of two variables}, J. Soviet Math. {\bf 35} (1986), 2809-2826.

\noindent [Ka2] V. N. Karpushkin, {\it Uniform estimates of oscillatory integrals with parabolic or 
hyperbolic phases}, J. Soviet Math. {\bf 33} (1986), 1159-1188.

\noindent [Ko] G. Kolesnik, {\it On the method of exponent pairs}, Acta Arith. {\bf 45} (1985) 115-143.

\noindent [Kr1] E. Kr\"{a}tzel, {\it Lattice points. Mathematics and its Applications}, East European Series {\bf 33}, Kluwer Academic Publishers Group, Dordrecht, 1988. 320 pp. ISBN: 90-277-2733-3.

\noindent [Kr2] E. Kr\"{a}tzel, {\it Analytische Funktionen in der Zahlentheorie}, [Analytic functions in number theory]
Teubner-Texte zur Mathematik [Teubner Texts in Mathematics], {\bf 139}. B. G. Teubner, Stuttgart, 2000. 288 pp. ISBN: 3-519-00289-2.

\noindent [L]  E. Landau, {\it \"Uber die Anzahl der Gitterpunkte in gewissen Bereichen}, Nachr. K\"onigl. Ges. Wiss.
G\"ottingen, math.-naturwiss. Kl. (1912), 687-771.

\noindent [Mu] W. M\"uller, {\it Lattice points in large convex bodies},  Monatsh. Math. {\bf 128} (1999), no. 4, 315-330. 

\noindent [N] A. Nadel, {\it Multiplier ideal sheaves and Kahler-Einstein metrics of positive scalar 
curvature},  Ann. of Math. (2)  {\bf 132}  no. 3 (1990), 549--596.

\noindent [NaSeW] A. Nagel, A. Seeger, and S. Wainger, {\it Averages over convex hypersurfaces},
Amer. J. Math. {\bf 115} (1993), no. 4, 903-927.

\noindent [PSSt] D. H. Phong, E. M. Stein, J. Sturm, {\it On the growth and stability of real-analytic functions}, Amer. J. Math, {\bf 121} (1999), 519-554.

\noindent [PSt] D. H. Phong, J. Sturm, {\it Algebraic estimates, stability of local zeta 
functions, and uniform estimates for distribution functions},  Ann. of Math. (2) {\bf 152}
(2000), no. 1, 277-329. 

\noindent [R1] B. Randol, {\it On the Fourier transform of the indicator function of a planar set}, Trans. Amer. Math. Soc. {\bf 139} (1969) 271-278.

\noindent [R2] B. Randol, {\it On the asymptotic behavior of the Fourier transform of the indicator function of a convex set}, Trans. Amer. Math. Soc. 
{\bf 139} (1969) 279-285. 

\noindent [R3] B. Randol, {\it A lattice point problem I}, Trans. Amer. Math. Soc. {\bf 121} (1966), 257-268.

\noindent [R4] B. Randol, {\it A lattice point problem II}, Trans. Amer. Math. Soc. {\bf 125} (1966), 101-113.

\noindent [Sc] H. Schulz, {\it Convex hypersurfaces of finite type and the asymptotics of their Fourier 
transforms}, Indiana Univ. Math. J. {\bf 40} (1991), no. 4, 1267--1275. 

\noindent [Si] W. Sierpi\'nski , {\it O pewnem zagadnieniu z rachunku funckcyi asymptotycnych}, Prace mat.-fiz. {\bf 17} (1906) 77-118.

\noindent [S1] E. Stein, {\it Harmonic analysis; real-variable methods, orthogonality, and oscillatory \hfill\break
integrals}, Princeton Mathematics Series {\bf 43}, Princeton University Press, Princeton, NJ, 1993.

\noindent [S2] E. Stein, {\it Maximal functions. I. Spherical means.} Proc. Nat. Acad. Sci. U.S.A. {\bf 73} (1976), no. 7, 2174-2175. 

\noindent [ShS] R. Shakarchi, E. Stein, {\it Functional analysis. Introduction to further topics in analysis}. Princeton Lectures in Analysis {\bf 4}. Princeton University Press, Princeton, NJ, 2011. xviii+423 pp. ISBN: 978-0-691-11387-6. 

\noindent [So] C. Sogge, {\it Maximal operators associated to hypersurfaces with one nonvanishing princi pal 
curvature} in {\it Fourier analysis and partial differential equations} (Miraflores de 
la Sierra, 1992),  317-323, Stud. Adv. Math., CRC, Boca Raton, FL, 1995. 

\noindent [SoS] C. Sogge and E. Stein, {\it Averages of functions over hypersurfaces in $\R^n$}, Invent. Math. {\bf 82} (1985), no. 3, 543-556.

\noindent [Sv] I. Svensson, {\it Estimates for the Fourier transform of the characteristic function of a convex set}, Ark. Mat. {\bf 9} (1971), 11-22.

\noindent [Sz] G. Szeg\"o, {\it Beitr\"age zur Theorie der Laguerreschen Polynome, II, Zahlentheoretische Anwendungen},
Math. Z. {\bf 25} (1926), 388-404.

\noindent [Tr] O. Trifonov, {\it On the number of lattice points in some two-dimensional domains}, Doklady Bulgar.
Akad. Nauk {\bf 41} (1988), 25-27.

\noindent [Va1] J. G. van der Corput, {\it \"{U}ber Gitterpunkte in der Ebene}, Math. Ann. {\bf 81} (1920) 1-20.

\noindent [Va2]  J. G. van der Corput, {\it Neue zahlentheortische Absch\"{a}tzungen}, Math. Ann. {\bf 89} (1923) 215-254.

\noindent [V] A. N. Varchenko, {\it Newton polyhedra and estimates of oscillatory integrals}, Functional Anal. Appl. {\bf 18} (1976), no. 3,  175-196.

\noindent [Vi] I. M. Vinogradov, {\it On the number of integer points in a sphere} (Russian), Izv. Akad. Nauk SSSR, Ser. Mat. {\bf 27} (1963), 957-968.

\end{document}